\documentclass[12pt]{article}
\usepackage{graphicx} 
\usepackage[english]{babel}
\usepackage{enumitem}
\usepackage{amsfonts}
\usepackage{latexsym} 
\usepackage{amsmath}
\usepackage{float}
\usepackage{geometry}
\usepackage{amsmath}
\usepackage{babel}
\usepackage{color}
\newtheorem{thm}{Theorem}
\newtheorem{lem}{Lemma}

\usepackage{url}
\begin{document}
\title{Multi-time-stepping splitting scheme with mass lumping for the multiscale wave equation}
  	
   \author{Yating Wang \thanks{School of Mathematics and Statistic, Xi'an Jiaotong University, Xi'an, People's Republic of China.}
   	  	\and Wing Tat Leung \thanks{Department of Mathematics, City University of Hong Kong, Hong Kong Special Administrative Region. }
  	\and Zhengya Yang \thanks{School of Mathematics and Statistic, Xi'an Jiaotong University, Xi'an, People's Republic of China.}
  	  \and Shufan Li \thanks{Department of Mathematics, City University of Hong Kong, Hong Kong Special Administrative Region.}}
  
 \maketitle

\begin{abstract}
    Efficient simulation of wave propagation in high-contrast media is challenging because standard explicit schemes are subject to severe time-step restrictions induced by fine spatial scales and fast multiscale modes, whereas fully implicit schemes require costly globally coupled solves. In this work, we develop a mass-lumped multi-time-stepping partially explicit splitting method for multiscale wave equations. Based on a nonlocal multicontinuum (NLMC) decomposition of the approximation space, the solution is separated into contrast-dependent fast components and contrast-independent slow components, which are advanced implicitly and explicitly, respectively. A mass-lumping approximation diagonalizes the coarse-scale mass operator and removes the coupling between two subproblems, thus enabling multirate time integration for the two parts. To further improve efficiency, we introduce a residual-based adaptive strategy that dynamically adjusts the time step of the explicitly treated subsystem according to its temporal error indicator. The proposed scheme is shown to be energy stable under suitable assumptions. Numerical experiments demonstrate that the proposed mass-lumped multi-time-stepping scheme substantially reduces computational cost while maintaining accuracy comparable to that of standard fully implicit methods.
\end{abstract}
\section{Introduction}
\label{section:Introduction}

Multiscale wave equations appear in wide applications such as hydrocarbon exploration, marine engineering, and geosciences \cite{multiscalewave1,multiscalewave2,multiscalewave3,multiscalewave4}. These models typically involve significant spatial variations in material or medium properties. These high contrast features poses considerable challenges for numerical simulation. It can induce numerical oscillations and demands sufficiently fine meshes to accurately resolve the underlying physical features, which substantially increases the computational cost.

A broad class of spatial multiscale methods, including multiscale finite element methods (MsFEM) \cite{msfem1,msfem2,msfem3}, generalized multiscale finite element methods (GMsFEM) \cite{gmsfem1,gmsfem2,gmsfem3}, localized orthogonal decomposiiton (LOD) \cite{lod1, lod2, lod3}, constraint energy-minimizing GMsFEM (CEM-GMsFEM) \cite{cem1,cem2,cem3,cem4}, and nonlocal multi-continua (NLMC) methods\cite{nlmc2,nlmc3}, has been developed to represent unresolved heterogeneities on coarse grids. For time-dependent multiscale problems, spatial model reduction alone does not necessarily remove the stiffness associated with fast solution components. In a standard explicit discretization, the admissible time step is constrained by the largest eigenvalues of the discrete operator. In wave problems, high-contrast velocity fields can impose severe stability restrictions, while high-frequency forcing demands fine temporal resolution, together leading to prohibitively small time steps. Fully implicit schemes relax this restriction, but require the solution of large and globally coupled linear systems at every time step. Moreover, the inverse of the associated matrix is generally nonlocal, introducing global numerical coupling despite the finite propagation speed of the underlying wave equation. These competing requirements motivate hybrid time discretizations that treat only the genuinely stiff components implicitly while retaining explicit updates for the remaining dynamics. 

The partially explicit splitting framework was first introduced in \cite{splitting1}. Its central idea is to decompose the multiscale approximation space into a contrast-dependent fast, or stiff, subspace and a contrast-independent slow subspace, which are advanced implicitly and explicitly, respectively. Under suitable conditions on the space decomposition, the remaining stability restriction is governed by the coarse-scale slow modes rather than by the coefficient contrast. This framework has subsequently been extended to nonlinear equations \cite{splitting1,splitting2,splitting4,splitting7}, fractional differential equations \cite{splitting3,splitting6}, and wave equations \cite{splitting5,li2025masslumping}. 

Though the partially explicit construction removes contrast dependence, the resulting subproblems remain coupled through the mass matrix for the wave equation. Consequently, the explicit component may still involve a coupled mass solve, and componentwise multirate integration is not straightforward. To overcome this bottleneck, mass-lumping approximation of the discrete $L^2$ inner product is needed \cite{li2025masslumping}. The resulting diagonal mass operator in the multiscale coordinates eliminates the cross-subspace mass coupling and avoids repeated mass-matrix inversions in the explicit update. 

Moreover, when the different components evolve on distinct characteristic time scales, it is inspiring to apply different time stepping on the mass-lumping partially explicit splitting scheme to further improve the efficiency. Multirate time integration has attracted increasing attention as an efficient strategy for systems whose components evolve on widely separated temporal scales \cite{multirate1,multirate2, Gallistl2024}. Instead of advancing the entire system with a single global time step, multirate methods assign different step sizes to individual subsystems according to their stiffness and dynamical activity. This structure is naturally compatible with partially explicit splitting, in which the fast and slow components receive different temporal treatments.% and therefore motivates the multirate partially explicit splitting framework developed in this work. %Moreover, a fixed multirate strategy may still incur substantial computational cost in long-time simulations or when stringent temporal accuracy is required.

In the present work, we construct a decomposition of the multiscale space into fast and slow subspaces and formulate a mass-lumped partially explicit scheme that decouples the resulting subproblems. Building on this formulation, we develop an adaptive multirate algorithm. The residual indicator quantifies the local temporal error in the subsystem and triggers refinement of its time step when necessary. Consequently, each component is advanced on a time scale commensurate with its dynamics, while contrast-independent stability is retained under the stated conditions and the temporal error is effectively controlled. We define an appropriate discrete energy and show the stability of multi-time-stepping strategy. Numerical experiments demonstrate that the combined use of mass lumping and adaptive multi-time-stepping substantially reduces the overall computational cost, particularly for long-time simulations, while maintaining comparable accuracy.

The layout of the paper is as follows. We present preliminaries in section \ref{section: Problem setting}. We give the construction of multiscale subspaces and the mass lumping method in section \ref{section: Construction of two multiscale subspaces}. In section \ref{section: Multirate partially explicit temporal splitting scheme}, we present the multirate method, analyze the stability result in subsections \ref{section: Stability13} and \ref{section: Stability2}, and present the adaptive multirate method in subsection \ref{section: Adaptive multirate time splitting algorithm}. Numerical tests are provided in section \ref{section: Numerical experiments}. Some concluding remarks are given in the final section.

\section{Problem setting}
\label{section: Problem setting}
Let $\Omega\subset\mathbb{R}^{d}$ be a bounded heterogeneous domain and $T>0$. We consider the following second order wave equation in $\Omega$ as
\begin{align}\label{B1}
\left\{\begin{aligned}
&\partial_{tt}u-\nabla\cdot(\kappa\nabla u)=f,\ \text{in}\ \Omega\times(0,T],\\
&u=0,\ \text{on}\ \partial\Omega\times(0,T],\\
&u(\cdot,0)=u_0(x),\ u_t(\cdot,0)=u_1(x),\ \text{at}\ t=0,
\end{aligned}\right.
\end{align}
where $\kappa\in L^{\infty}(\Omega)$ is a multiscale high contrast field. The initial conditions $u_0(x)$, $u_1(x)$ and the source term $f$ are given functions with suitable regularity.

The weak formulation of the problem \eqref{B1} reads as follows: find $u(\cdot,t)\in H_0^1(\Omega)$, such that
\begin{align}\label{B2}
(\partial_{tt}u,v)+a(u,v)=(f,v),\ \forall v\in H_0^1(\Omega),
\end{align}
where the bilinear form $a(\cdot,\cdot)$ is given by
\begin{align*}
a(u,v)=\int_{\Omega}\kappa\nabla u\cdot\nabla v,
\end{align*}
and we define the energy norm as $\Vert u\Vert_a=a(u,u)^{\frac{1}{2}}$.

Let $\mathcal{T}_{H}$ be a coarse partition of the computational domain $\Omega$ with mesh size $H$. We will construct multiscale basis functions on $\mathcal{T}_{H}$ and form the multiscale space $V_{H}\subset H_{0}^{1}(\Omega)$, the semi-discretization in space then leads to seek the approximation solution $u_{H}(\cdot, t)\in V_{H}$, such that
\begin{align}\label{B3}
(\partial_{tt}u_{H}(t),v)+a(u_{H}(t),v)=(f,v),\ \forall v\in V_{H}.
\end{align}

For the temporal discretization, we consider a fixed time step $\Delta \tau=T/N$ where $N$ is the number of time steps, and the second order central difference scheme in time reads
\begin{align*}
\left(\frac{u_{H}^{n+1}-2u_{H}^{n}+u_{H}^{n-1}}{\Delta \tau^2},v\right)+\frac{1}{4}a(u_{H}^{n+1}+2u_{H}^{n}+u_{H}^{n-1},v)=(f^{n+1},v),\ \forall v\in V_{H},
\end{align*}
for all $n=1,2,\cdots,N-1$, where $u_H^n\approx u(\cdot,t^n)$. It is well known that this scheme is unconditionally stable.

\section{Partially explicit splitting scheme with mass lumping}
\label{section: Construction of two multiscale subspaces}

Now we assume that the approximation space $V_H$ can be decomposed into two subspaces namely 
\begin{align*}
V_H=V_{H,1}+V_{H,2},
\end{align*}
then a partially explicit temporal splitting scheme reads: finding $\{u_{H,1}^{n+1}\}_{n=0}^{N-1}\in V_{H,1}$ and $\{u_{H,2}^{n+1}\}_{n=0}^{N-1}\in V_{H,2}$, such that for all $v_{1}\in V_{H,1}$ and $v_2\in V_{H,2}$, we have
\begin{align}\label{B4}
\begin{aligned}
&\left(\frac{u_{H}^{n+1}-2u_{H}^{n}+u_{H}^{n-1}}{\Delta \tau^2},v_1\right)+\frac{1}{2}a(u_{H,1}^{n+1}+u_{H,1}^{n-1}+2u_{H,2}^{n},v_1)=(f^{n+1},v_1),\\
&\left(\frac{u_{H}^{n+1}-2u_{H}^{n}+u_{H}^{n-1}}{\Delta \tau^2},v_2\right)+a(u_{H,1}^{n}+u_{H,2}^{n},v_2)=(f^{n+1},v_2).
\end{aligned}
\end{align}
Thus the solution at time step $n+1$ will be $u_{H}^{n+1}=u_{H,1}^{n+1}+u_{H,2}^{n+1}$. It has been proved in \cite{splitting5} that under suitable choice of multiscale spaces $V_{H,1}$ and $V_{H,2}$, the scheme \eqref{B4} are stable and satisfy the energy conservation. However, the system is still coupled. To decouple the system, in this section, we will describe the construction of the multiscale finite element spaces such that the two equations can be solved separately by the mass lumping technique. 

We first construct the auxiliary spaces $V_{\text{aux,1}}$ and $V_{aux,2}$, then utilize them to construct the multiscale subspaces $V_{H,1}$ and $V_{H,2}$. Let $\{K_i\}_{i=1}^{N_c}$ be the set of coarse elements in $\mathcal{T}_{H}$ and $N_c$ be the number of coarse elements. Denote by $K_i^+$ an oversampling region obtained by enlarging $K_i$ a few coarse layers. Denote by $V=H_0^1(\Omega)$, and $V(K_i)$ the restriction of $V$ on $K_i$. 

%We first construct the auxiliary spaces $V_{\text{aux},1}$ and $V_{\text{aux},2}$. 
For each coarse element $K_i\in\mathcal{T}_{H}$, we consider 
\begin{align*}
K_1^{(i)}=\{x\in K_i\mid\kappa(x)\geq \kappa^*\},\ K_2^{(i)}=\{K_i \backslash K_1^{(i)}\},
\end{align*}
where $\kappa^*$ is a constant separating the high wave speed and low wave speed. We define the auxiliary basis functions of $V_{H,1}$ as $\psi_j^{(i)}=I_{K_j^{(i)}}(x), \;\; j=1,2,$ where $I_{K_j^{(i)}}$ denote the characteristic function of $K_j^{(i)}$. Then we define the first auxiliary space as 
\begin{align*}
V_{\text{aux},1}=\text{span}\{\psi_1^{(i)}, i = 1,\cdots,N_c\}.
\end{align*}

To obtain the auxiliary basis functions for $V_{H,2}$, we let $\widetilde{V}=\{v\in V\mid (v,\psi_j^{(i)})=0\}$, then find $\xi_j^{(i)} \in V(K_i)\cap\widetilde{V}$ and $\gamma_j^{(i)}\in\mathbb{R}$, such that $\forall v\in V(K_i)\cap\widetilde{V}$
\begin{align}\label{C1}
\int_{K_i}\kappa\nabla\xi_j^{(i)}\cdot\nabla v=\gamma_j^{(i)}\int_{K_i}\xi_j^{(i)}v.
\end{align}
Take the first $l_i$ eigenfunctions $\xi_j^{(i)}$, then the auxiliary space $V_{\text{aux},2}$ is 
\begin{align*}
V_{\text{aux},2}=\text{span}_{i,j}\{\psi_2^{(i)}, \xi_j^{(i)}, i = 1,\cdots,N_c, j=1,\cdots, l_i\}.
\end{align*}

Next, we construct the multiscale subspaces $V_{H,1}$ and $V_{H,2}$. We solve the following problem to obtain the multiscale basis functions $\phi_{j,1}^{(i)}$ of $V_{H,1}$: find $(\phi_{j,1}^{(i)},\mu_{j,1}^{(i)})\in V_0(K_i^+)\times(V_{\text{aux},1}+V_{\text{aux},2})$, such that
\begin{align}\label{C2}
\left\{\begin{aligned}
a(\phi_{j,1}^{(i)},v)+(\mu_{j,1}^{(i)},v)=&0,\ \forall v\in V_0(K_i^+),\\
(\phi_{j,1}^{(i)},\theta_1)=&(\psi_j^{(i)},\theta_1), \ \forall \theta_1 \in V_{\text{aux},1}, \\%&\delta_{il}\delta_{jk},\\
(\phi_{j,1}^{(i)},\theta_2)=&0, \ \forall \theta_2 \in V_{\text{aux},2}.
\end{aligned}\right.
\end{align}
Similarly, we solve the following problem to obtain the multiscale basis $\phi_{j,2}^{(i)}$ of $V_{H,2}$: find $(\phi_{j,2}^{(i)},\mu_{j,2}^{(i)})\in V_0(K_i^+)\times(V_{\text{aux},1}+V_{\text{aux},2})$, such that
\begin{align}\label{C3}
\left\{\begin{aligned}
a(\phi_{j,2}^{(i)},v)+(\mu_{j,2}^{(i)},v)=&0,\ \forall v\in V_0(K_i^+),\\
(\phi_{j,2}^{(i)},\theta_1)=&0, \ \forall \theta_1 \in V_{\text{aux},1},\\
(\phi_{j,2}^{(i)},\theta_2)=&(\xi_k^{(l)},\theta_2),\ \forall \theta_2 \in V_{\text{aux},2}.
\end{aligned}\right.
\end{align}
Finally, the multiscale finite element spaces $V_{H,1}$ and $V_{H,2}$ are defined as follows
\begin{align}\label{C4}
V_{H,1}=\text{span}_{i,j}\{\phi_{j,1}^{(i)}\},\ V_{H,2}=\text{span}_{i,j}\{\phi_{j,2}^{(i)}\}.
\end{align}
We remark that the basis in $V_{H,2}$ that correspond to the eigenfunctions $\xi_j^{(i)}$ are complementary and can improve the accuracy. For $u_H\in V_{H,1}+V_{H,2}$, we write $u_H=\sum_{i,j}u_{j,1}^{(i)}\phi_{j,1}^{(i)}+\sum_{i,j}u_{j,2}^{(i)}\phi_{j,2}^{(i)}$, then 
\begin{align*}
a(u_H,v)=-\left(\sum_{i,j}u_{j,1}^{(i)}\mu_{j,1}^{(i)},v\right)-\left(\sum_{i,j}u_{j,2}^{(i)}\mu_{j,2}^{(i)},v\right).
\end{align*}
Let $W_H = (V_{\text{aux},1}\oplus V_{\text{aux},2})$, for all $z\in W_H$, the weak formulation can be written as
\begin{align*}
(\sum_{i,j}(u_{j,1}^{(i)})_{tt}\phi_{j,1}^{(i)}+\sum_{i,j}(u_{j,2}^{(i)})_{tt}\phi_{j,2}^{(i)},z)=(\sum_{i,j}u_{j,1}^{(i)}\mu_{j,1}^{(i)}+\sum_{i,j}u_{j,2}^{(i)}\mu_{j,2}^{(i)}+f,z),
\end{align*}
 Take $z = \psi_k^{(l)}, \xi_k^{(l)}$ and use the properties of the auxiliary spaces, i.e.
% we have
% \begin{align*}
% (\sum_{i,j}(u_{j,1}^{(i)})_{tt}\phi_{j,1}^{(i)}+\sum_{i,j}(u_{j,2}^{(i)})_{tt}\phi_{j,2}^{(i)},\psi_k^{(l)})=(\sum_{i,j}u_{j,1}^{(i)}\mu_{j,1}^{(i)}+\sum_{i,j}u_{j,2}^{(i)}\mu_{j,2}^{(i)}+f,\psi_k^{(l)}),\\
% (\sum_{i,j}(u_{j,1}^{(i)})_{tt}\phi_{j,1}^{(i)}+\sum_{i,j}(u_{j,2}^{(i)})_{tt}\phi_{j,2}^{(i)},\xi_k^{(l)})=(\sum_{i,j}u_{j,1}^{(i)}\mu_{j,1}^{(i)}+\sum_{i,j}u_{j,2}^{(i)}\mu_{j,2}^{(i)}+f,\xi_k^{(l)}).
% \end{align*}
\begin{align*}
(\phi_{j,1}^{(i)},\psi_k^{(l)})=\delta_{il}\delta_{jk},\ (\phi_{j,1}^{(i)},\xi_k^{(l)})=0, (\phi_{j,2}^{(i)},\ \psi_k^{(l)})=0,\ (\phi_{j,1}^{(i)},\xi_k^{(l)})=\delta_{il}\delta_{jk},
\end{align*}
 we get
\begin{align*}
(u_{k,1}^{(l)})_{tt}=(\sum_{i,j}u_{j,1}^{(i)}\mu_{j,1}^{(i)}+\sum_{i,j}u_{j,2}^{(i)}\mu_{j,2}^{(i)}+f,\psi_k^{(l)}),\\
(u_{k,2}^{(l)})_{tt}=(\sum_{i,j}u_{j,1}^{(i)}\mu_{j,1}^{(i)}+\sum_{i,j}u_{j,2}^{(i)}\mu_{j,2}^{(i)}+f,\xi_k^{(l)}).
\end{align*}
Since 
\begin{align*}
(\sum_{i,j}u_{j,1}^{(i)}\mu_{j,1}^{(i)}+\sum_{i,j}u_{j,2}^{(i)}\mu_{j,2}^{(i)},\psi_k^{(l)})=(\sum_{i,j}u_{j,1}^{(i)}\mu_{j,1}^{(i)}+\sum_{i,j}u_{j,2}^{(i)}\mu_{j,2}^{(i)},\phi_{k,1}^{(l)}),\\%=-a(u_H,\phi_{k,1}^{(l)}),
(\sum_{i,j}u_{j,1}^{(i)}\mu_{j,1}^{(i)}+\sum_{i,j}u_{j,2}^{(i)}\mu_{j,2}^{(i)},\xi_k^{(l)})=(\sum_{i,j}u_{j,1}^{(i)}\mu_{j,1}^{(i)}+\sum_{i,j}u_{j,2}^{(i)}\mu_{j,2}^{(i)},\phi_{k,2}^{(l)}),%=-a(u_H,\phi_{k,2}^{(l)})
\end{align*}
one can obtain the following schemes
\begin{align*}
(u_{k,1}^{(l)})_{tt}+\sum_{i,j}u_{j,1}^{(i)}a(\phi_{j,1}^{(i)},\phi_{k,1}^{(l)})+\sum_{i,j}u_{j,2}^{(i)}a(\phi_{j,2}^{(i)},\phi_{k,1}^{(l)})=(f,\psi_k^{(l)}),\ \forall l,k,\\
(u_{k,2}^{(l)})_{tt}+\sum_{i,j}u_{j,1}^{(i)}a(\phi_{j,1}^{(i)},\phi_{k,2}^{(l)})+\sum_{i,j}u_{j,2}^{(i)}a(\phi_{j,2}^{(i)},\phi_{k,2}^{(l)})=(f,\xi_k^{(l)}),\ \forall l,k.
\end{align*}

In this case, one obtains a diagonal mass matrix. Intuitively, we will approximate the original mass matrix $M$ in \eqref{B4} with $\widetilde{M}=(\Pi(\phi_i),\Pi(\phi_j))$, where $\Pi:V_H\rightarrow V_{\text{aux},1}\oplus V_{\text{aux},2}$ is the $L^2$ projection operator. 
We denote $b(u,v)=(\Pi u,\Pi v)$. Then $\forall v_1 \in V_{H,1}, \forall v_2 \in V_{H,2}$,the decoupled splitting scheme corresponding to \eqref{B4} reads
\begin{align}\label{C5}
\begin{aligned}
&b\left(\frac{u_{H,1}^{n+1}-2u_{H,1}^{n}+u_{H,1}^{n-1}}{\Delta \tau^2},v_1\right)+\frac{1}{2}a(u_{H,1}^{n+1}+2u_{H,2}^{n}+u_{H,1}^{n-1},v_1)=b(f^{n+1},v_1),\\
&b\left(\frac{u_{H,2}^{n+1}-2u_{H,2}^{n}+u_{H,2}^{n-1}}{\Delta \tau^2},v_2\right)+a(u_{H,1}^{n}+u_{H,2}^{n},v_2)=b(f^{n+1},v_2).
\end{aligned}
\end{align}

\section{Multirate partially explicit temporal splitting scheme}
\label{section: Multirate partially explicit temporal splitting scheme}
In this section, we will introduce a multirate partially explicit scheme based on the multiscale subspaces constructed in Section \ref{section: Construction of two multiscale subspaces}. We consider the coarse time step size $\Delta T$ and fine time step size $\Delta t$ with $\Delta T=m\Delta t$, and each coarse time interval can be written as $(T_k,T_{k+1}]=\bigcup_{j=n_k}^{n_{k+1}-1}(t_j,t_{j+1}]$, with $n_k=mk$.
% and denote the fine partition and coarse partition of the whole time domain $(0,T]$ as
% \begin{align*}
% 0=t_0<t_1<\cdots<t_{N-1}=T 
% \end{align*}
% and
% \begin{align*}
% 0=T_0<T_1<\cdots<T_{\frac{N-1}{m}}=T.
% \end{align*}
% Therefore we can write each coarse time interval as 
 We will omit the right hand side $f$ later on. The multirate partially explicit temporal splitting scheme is defined as follows. For each coarse interval $(T_k,T_{k+1}]$, we look for $u_{H}^{n_{k+1}}=u_{H,1}^{n_{k+1}}+u_{H,2}^{n_{k+1}}$ with a given $u_{H}^{n_{k}}=u_{H,1}^{n_{k}}+u_{H,2}^{n_{k}}$ from one of the following cases:\\

\textbf{Case 1: coarse-coarse}% Take $\Delta \tau = \Delta T$ in both equations in \eqref{C5}.
\begin{align}\label{D1}
\begin{aligned}
&b\left(\frac{u_{H,1}^{n_{k+1}}-2u_{H,1}^{n_k}+u_{H,1}^{n_{k-1}}}{\Delta T^2},v_1\right)+\frac{1}{2}a(u_{H,1}^{n_{k+1}}+u_{H,1}^{n_{k-1}}+2u_{H,2}^{n_{k}},v_1)=0,\\
&b\left(\frac{u_{H,2}^{n_{k+1}}-2u_{H,2}^{n_k}+u_{H,2}^{n_{k-1}}}{\Delta T^2},v_2\right)+a(u_{H,1}^{n_k}+u_{H,2}^{n_k},v_2)=0,
\end{aligned}
\end{align}
$\forall v_1\in V_{H,1}$ and $\forall v_2\in V_{H,2}$.\\

\textbf{Case 2: coarse-fine} %Take $\Delta \tau = \Delta T$ in the first equation of \eqref{C5}, and $\Delta \tau = \Delta t$ in the second equation of \eqref{C5}.
\begin{align}\label{D2}
\begin{aligned}
&b\left(\frac{u_{H,1}^{n_{k+1}}-2u_{H,1}^{n_k}+u_{H,1}^{n_{k-1}}}{\Delta T^2},v_1\right)+\frac{1}{2}a(u_{H,1}^{n_{k+1}}+u_{H,1}^{n_{k-1}}+2u_{H,2}^{n_{k}},v_1)=0,\\
&b\left(\frac{u_{H,2}^{n+1}-2u_{H,2}^{n}+u_{H,2}^{n-1}}{\Delta t^2},v_2\right)+a(u_{H,1}^{n_k}+u_{H,2}^{n},v_2)=0,
\end{aligned}
\end{align}
$\forall v_1\in V_{H,1}$ and $\forall v_2\in V_{H,2}$, $n=n_k, \cdots, n_k+m-1$.\\

% \textbf{Case 3: fine-coarse}
% \begin{align}\label{D3}
% \begin{aligned}
% &b\left(\frac{u_{H,1}^{n+1}-2u_{H,1}^{n}+u_{H,1}^{n-1}}{\Delta t^2},v_1\right)+\frac{1}{2}a(u_{H,1}^{n+1}+u_{H,1}^{n-1}+2u_{H,2}^{n_{k}},v_1)=b(f^{n+1},v_1),\\
% &b\left(\frac{u_{H,2}^{n_{k+1}}-2u_{H,2}^{n_k}+u_{H,2}^{n_{k-1}}}{\Delta T^2},v_2\right)+a(u_{H,1}^{n_k}+u_{H,2}^{n_k},v_2)=b(f^{n_{k+1}},v_2),
% \end{aligned}
% \end{align}
% $\forall v_1\in V_{H,1}$ and $\forall v_2\in V_{H,2}$.\\

\textbf{Case 3: fine-fine}: Take $\Delta \tau = \Delta t$ in both equations of \eqref{C5}.
% \begin{align}\label{D4}
% \begin{aligned}
% &b\left(\frac{u_{H,1}^{n+1}-2u_{H,1}^{n}+u_{H,1}^{n-1}}{\Delta t^2},v_1\right)+\frac{1}{2}a(u_{H,1}^{n+1}+u_{H,1}^{n-1}+2u_{H,2}^{n},v_1)=0,\\%b(f^{n+1},v_1)
% &b\left(\frac{u_{H,2}^{n+1}-2u_{H,2}^{n}+u_{H,2}^{n-1}}{\Delta t^2},v_2\right)+a(u_{H,1}^{n}+u_{H,2}^{n},v_2)=0,%b(f^{n+1},v_2)
% \end{aligned}
% \end{align}
% $\forall v_1\in V_{H,1}$ and $\forall v_2\in V_{H,2}$.

We remark that the scheme marches with the coarse time step size $\Delta T$ first, and we determine whether $u_{H,1}$ or $u_{H,2}$ needs to be solved using the fine time step size $\Delta t$ in each coarse time instance by an adaptive method introduced in section \ref{section: Adaptive multirate time splitting algorithm}. In the next two subsections, we will discuss the stability of the three cases.

\subsection{Stability of Case 1 or Case 3} \label{section: Stability13}

Let $\gamma_a$ be a constant defined by 
\begin{align}\label{eq:gamma}
\gamma_a:=\sup_{v_1\in V_{H,1}, v_2\in V_{H,2}}\frac{|a(v_1,v_2)|}{\Vert v_1\Vert_a\Vert v_2\Vert_a}<1.
\end{align}
Following the same idea as in \cite{splitting5}, we can get the desired result and we remark that the stability for Case 1 is similar as Case 3.
\begin{lem}\label{l1}
For the syetem \eqref{D1}, we define the discrete energy $E^{n_{k+\frac{1}{2}}}$ as
\begin{align*}
\begin{aligned}
E^{n_{k+\frac{1}{2}}}=&\sum_{i=1}^{2}\Vert u_{H,i}^{n_{k+1}}-u_{H,i}^{n_k}\Vert_b^2+\frac{\Delta T^2}{2}\sum_{i=1}^{2}\left(\Vert u_{H,i}^{n_{k+1}}\Vert_a^2+\Vert u_{H,i}^{n_k}\Vert_a^2 \right)\\
&+\Delta T^2\left(a(u_{H,2}^{n_k},u_{H,1}^{n_{k+1}})+a(u_{H,1}^{n_k},u_{H,2}^{n_{k+1}})\right)-\frac{\Delta T^2}{2}\Vert u_{H,2}^{n_{k+1}}-u_{H,2}^{n_k}\Vert_a^2,
\end{aligned}
\end{align*}
where $\Vert u\Vert_b^2=b(u,u)$, then we have energy conversation $E^{n_{k+\frac{1}{2}}}=E^{n_{k-\frac{1}{2}}}$.
Moreover, the scheme is stable if
\begin{align}\label{D9}
\Vert v\Vert_b^2\geq\frac{\Delta T^2}{2}\Vert v\Vert_a^2,\ \forall v\in V_{H,2},
\end{align}
and we have $\sum_{i=1,2}\left(\Vert u_{H,i}^{n_{k+1}}\Vert_a^2+\Vert u_{H,i}^{n_k}\Vert_a^2\right)\leq\frac{2}{\Delta T^2(1-\gamma_a)}E^{n_{\frac{1}{2}}}.$
\end{lem}

\subsection{Stability of Case 2} \label{section: Stability2}

\begin{thm}\label{l2}
For the syetem \eqref{D2}, we define the energy as 
\begin{align}
\tilde{E}^{n_{k+\frac{1}{2}}}:=\sum_{i=1,2}\|u_{i}^{n_{k+1}}-u_{i}^{n_k}\|_{b}^{2}+\cfrac{\Delta T^2}{2}(\|u_{1}^{n_k}\|_{a}^{2}+\|u_{1}^{n_{k+1}}\|_{a}^{2})\\
+\Delta T^2(a(u_{1}^{n_{k+1}},u_{1}^{n_k})+a(u_{2}^{n_{k+1}},u_{1}^{n_k})+a(u_{2}^{n_k},u_{1}^{n_{k+1}})),
\end{align}
we have $ \tilde{E}^{n_{k+\frac{1}{2}}} \leq \tilde{E}^{n_{k-\frac{1}{2}}}.$
Similarly, the scheme \eqref{D2} is stable under the condition \eqref{D9}, and we have
\begin{align*}
\sum_{i=1,2}\left(\Vert u_{H,i}^{n_{k+1}}\Vert_a^2+\Vert u_{H,i}^{n_k}\Vert_a^2\right)\leq\frac{2}{\Delta T^2(1-\gamma_a)}E^{n_{\frac{1}{2}}}.
\end{align*}
\end{thm}

%\begin{proof}
\textit{Proof}: For the system \eqref{D2}, in a matrix form, we write
\begin{align} \label{eq:case2-1-u2}
\begin{aligned}
    \Delta_{tt} u_1^{n_k} + A_{11}(\frac{u_1^{n_{k+1}}+u_1^{n_{k-1}}}{2}) + A_{12} u_2^{n_k} &= 0,\\
    \Delta_{tt} u_2^n + A_{21} u_1^{n_k} + A_{22}u_2^n &= 0,
    \end{aligned}
\end{align}
where $n=n_k+1, \cdots, n_{k+1}$.

Suppose $A_{22}\phi_{2,i}=\lambda_{2,i}\phi_{2,i}$, then we can write
\begin{equation*}
    u_{2}^{n}=\sum_{i}u_{2,i}^{n}\phi_{2,i}, \;\; \Pi_{2,1}u_{1}^{n_k}=\sum\tilde{u}_{1,i}^{n_k}\phi_{2,i}
\end{equation*}

The temporal discretization of the second equation in \eqref{eq:case2-1-u2} leads to
\begin{equation}\label{eq:case2-1-u2_2}
  u_{2,i}^{n+1}-2u_{2,i}^{n}+u_{2,i}^{n-1}+\tilde{\lambda}_{2,i}\bar{u}^n_{1,i}+\tilde{\lambda}_{2,i}u_{2,i}^{n}=0  
\end{equation}
where $\tilde{\lambda}_{2,i}=\Delta t^{2}\lambda_{2,i}$, and $\bar{u}^{n+m}_{1,i} = \tilde{u}^{n_k}_{1,i}+ \cfrac{m}{M}(\tilde{u}^{n_{k+1}}_{1,i}-\tilde{u}^{n_k}_{1,i})$.

Equation \eqref{eq:case2-1-u2_2} can be further written as
\begin{equation*}
    \left(\begin{array}{c}
u_{2,i}^{n+1}\\
u_{2,i}^{n}
\end{array}\right)=\left(\begin{array}{cc}
2-\tilde{\lambda}_{2,i} & -1\\
1 & 0
\end{array}\right)\left(\begin{array}{c}
u_{2,i}^{n}\\
u_{2,i}^{n-1}
\end{array}\right) -\left(\begin{array}{c}
\tilde{\lambda}_{2,i}  \bar{u}^n_{1,i}\\
0
\end{array}\right)
\end{equation*}

Let $C=\left(\begin{array}{cc}
2-\tilde{\lambda}_{2,i} & -1\\
1 & 0
\end{array}\right)$, 
it has two eigenvalues $a+b\textbf{i}$ and $a-b\textbf{i}$, where $a=({2-\tilde{\lambda}_{2,i}})/2$ and $b=\sqrt{4-(2-\tilde{\lambda}_{2,i})^2}/\sqrt{2}$. Write $a=\cos(2\theta_i), b=\sin(2\theta_i)$ where $\theta_{i}=\cfrac{1}{2}\arctan(\sqrt{4-(2-\tilde{\lambda}_{2,i})^2}/(2-\tilde{\lambda}_{2,i}))$, and let
\begin{equation}
    R_{2,i}=\left(\begin{array}{cc}
a & -b\\
b & a
\end{array}\right), \;\; P_{2,i}=\cfrac{1}{2}\left(\begin{array}{cc}
\frac{1}{ \cos(\theta_i)} & \frac{1}{ \sin(\theta_i)}\\
-\frac{1}{ \cos(\theta_i)} &\frac{1}{ \sin(\theta_i)}
\end{array}\right).
\end{equation}
% then \begin{equation*}
%     P_{2,i}^{-1}=\left(\begin{array}{cc}
% \cos(\theta_i) & -\cos(\theta_i)\\
% \sin(\theta_i) & \sin(\theta_i)
% \end{array}\right).
% \end{equation*}
Let $E_{2,i}= P_{2,i}^{-1}$, we have $P_{2,i}^{-1} C P_{2,i}=R_{2,i}$, or equivalently, $C=E_{2,i}^{-1}R_{2,i}E_{2,i}$. That is to say,
\begin{equation*}
    \left(\begin{array}{c}
u_{2,i}^{n+1}\\
u_{2,i}^{n}
\end{array}\right)=E_{2,i}^{-1}R_{2,i} E_{2,i}\left(\begin{array}{c}
u_{2,i}^{n}\\
u_{2,i}^{n-1}
\end{array}\right) -\left(\begin{array}{c}
\tilde{\lambda}_{2,i} \bar{u}^n_{1,i}\\
0
\end{array}\right).
\end{equation*}
% which implies
% \begin{equation*}
% E_{2,i}\left(\begin{array}{c}
% u_{2,i}^{n+1}\\
% u_{2,i}^{n}
% \end{array}\right)=R_{2,i}E_{2,i}\left(\begin{array}{c}
% u_{2,i}^{n}\\
% u_{2,i}^{n-1}
% \end{array}\right)  - \tilde{\lambda}_{2,i}  \bar{u}_{1,i}^{n} \left(\begin{array}{c}
% \cos\theta_i \\
% \sin\theta_i 
% \end{array}\right)
% \end{equation*}
% %$\bar{u}^{n+m}_{1,i} = \tilde{u}^{n_k}_{1,i}+ \cfrac{m}{M}(\tilde{u}^{n_{k+1}}_{1,i}-\tilde{u}^{n_k}_{1,i})$
Marching forward $m$ steps, we get
\begin{align*}
&E_{2,i}\left(\begin{array}{c}
u_{2,i}^{n+m}\\
u_{2,i}^{n+m-1}
\end{array}\right)\\
&=R_{2,i}^{m}E_{2,i}\left(\begin{array}{c}
u_{2,i}^{n}\\
u_{2,i}^{n-1}
\end{array}\right)-\tilde{\lambda}_{2,i}\sum_{j=0}^{m-1}R_{2,i}^{j} (\tilde{u}^{n_k}_{1,i}+ \cfrac{j}{M}(\tilde{u}^{n_{k+1}}_{1,i}-\tilde{u}^{n_k}_{1,i})) \left(\begin{array}{c}
\cos\theta_{i}\\
\sin\theta_{i}
\end{array}\right)\\
 & =R_{2,i}^{m}E_{2,i}\left(\begin{array}{c}
u_{2,i}^{p}\\
u_{2,i}^{p-1}
\end{array}\right)-\tilde{\lambda}_{2,i}\tilde{u}_{1,i}^{n_k}(I-R_{2,i})^{-1}(I-R_{2,i}^{m})\left(\begin{array}{c}
\cos\theta_{i}\\
\sin\theta_{i}
\end{array}\right)\\
&- \cfrac{\tilde{\lambda}_{2,i}}{M}(\tilde{u}^{n_{k+1}}_{1,i}-\tilde{u}^{n_k}_{1,i})(I-R_{2,i})^{-2}(R_{2,i}-mR_{2,i}^{m}+(m-1)R_{2,i}^{m+1})\left(\begin{array}{c}
\cos\theta_{i}\\
\sin\theta_{i}
\end{array}\right)
\end{align*}

One can check that
\begin{align*}
 & E_{2,i}^{-1} R_{2,i}^{m}E_{2,i} = \cfrac{1}{\sin(2\theta_i )}\left(\begin{array}{cc}
\sin2(m+1)\theta_{i} & -\sin2m\theta_{i}\\
 \sin2m\theta_{i} & -\sin2(m-1)\theta_{i}
\end{array}\right),
\end{align*}
\begin{align*}
  & E^{-1}_{2,i} (I-R_{2,i})^{-1}(I-R_{2,i}^{m})\left(\begin{array}{c}
\cos\theta_{i}\\
\sin\theta_{i}
\end{array}\right)\\
&=\cfrac{1}{2\sin2\theta_{i}(1-\cos2\theta_{i})}
\left(\begin{array}{c}
\sin 2\theta_i -  \sin (2(m+1)\theta_{i})+\sin (2m\theta_{i})\\
\sin 2\theta_i + \sin (2(m-1)\theta_{i})-\sin (2m\theta_{i})
\end{array}\right),
\end{align*}
\begin{align*}
 &E^{-1}_{2,i} (I-R_{2,i})^{-2}(R_{2,i}-mR_{2,i}^{m}+(m-1)R_{2,i}^{m+1})\left(\begin{array}{c}
\cos\theta_{i}\\
\sin\theta_{i}
\end{array}\right) \\
& =\cfrac{1}{\sin2\theta_{i}\tilde{\lambda}_{2,i}}
\left(\begin{array}{c}
-(m-1)\sin(2(m + 1)\theta_i ) -\sin(2\theta_i ) + m\sin(2m\theta_i ) \\
-(m-1)\sin(2m\theta_i)  + m\sin(2(m - 1)\theta_i)
\end{array}\right),
\end{align*}
where we use the fact that $\cfrac{1}{2(1-\cos2\theta_{i})}=\cfrac{1}{\tilde{\lambda}_{2,i}}$. We then have 
\begin{align*}
&\left(\begin{array}{c}
u_{2,i}^{n+m}\\
u_{2,i}^{n+m-1}
\end{array}\right)
=  \cfrac{1}{\sin2\theta_{i}}\left(\begin{array}{cc}
\sin(2(m+1)\theta_{i}) & -\sin(2m\theta_{i})\\
\sin(2m\theta_{i}) & -\sin(2(m-1)\theta_{i})
\end{array}\right)\left(\begin{array}{c}
u_{2,i}^{n}\\
u_{2,i}^{n-1}
\end{array}\right)\\
&-\cfrac{1}{\sin2\theta_{i}}\left(\begin{array}{c}
\sin 2\theta_i -  \sin (2(m+1)\theta_{i})+\sin (2m\theta_{i})\\
\sin 2\theta_i + \sin (2(m-1)\theta_{i})-\sin (2m\theta_{i})
\end{array}\right)\tilde{u}_{1,i}^{n_k}-\cfrac{1}{M\sin2\theta_{i}}\\
&\left(\begin{array}{c}
-(m-1)\sin(2(m + 1)\theta_i ) -\sin(2\theta_i ) + m\sin(2m\theta_i ) \\
-(m-1)\sin(2m\theta_i)  + m\sin(2(m - 1)\theta_i)
\end{array}\right)(\tilde{u}^{n_{k+1}}_{1,i}-\tilde{u}^{n_k}_{1,i}).
\end{align*}

With a simplified notation, we can write
\begin{align}
u_{2,i}^{n+m}-u_{2,i}^{n+m-1}= 
% & \cfrac{\sin(2(m+1)\theta_{i})-\sin(2m\theta_{i})}{\sin2\theta_{i}}u_{2,i}^{n}-\cfrac{\sin(2m\theta_{i})-\sin(2(m-1)\theta_{i})}{\sin2\theta_{i}}u_{2,i}^{n-1}\\
%  & -\cfrac{2\sin(2m\theta_{i})-\sin(2(m+1)\theta_{i})-\sin(2(m-1)\theta_{i})}{\sin 2\theta_{i}}\tilde{u}_{1,i}^{n_k}\\
%  &-\cfrac{-(m-1)\sin(2(m+1)\theta_{i})+(2m-1)\sin(2m\theta_{i})-m\sin(2(m-1)\theta_{i})-\sin 2\theta_{i}}{\sin 2\theta_{i}}(\tilde{u}^{n_{k+1}}_{1,i}-\tilde{u}^{n_k}_{1,i}) \\
& \alpha_{i,m+1}u_{2,i}^{n} -\alpha_{i,m}u_{2,i}^{n-1}+\gamma_{i,m} \tilde{u}_{1,i}^{n_k} -\mu_{i,m} (\tilde{u}^{n_{k+1}}_{1,i}-\tilde{u}^{n_k}_{1,i}) \label{eq:case2_1-bsum}
\end{align}
where
\begin{align*}
&\alpha_{i,m+1}  :=\cfrac{\sin(2(m+1)\theta_{i})-\sin(2m\theta_{i})}{\sin2\theta_{i}}, \\
&\gamma_{i,m}  :=\cfrac{\sin(2(m+1)\theta_{i})-2\sin(2m\theta_{i})+\sin(2(m-1)\theta_{i})}{\sin2\theta_{i}}=\alpha_{i,m+1}-\alpha_{i,m},\\
&\mu_{i,m} := \cfrac{(2m-1)\sin2m\theta_{i}-(m-1)\sin2(m+1)\theta_{i}-m\sin2(m-1)\theta_{i}}{M\sin 2\theta_{i}}-\cfrac{1}{M}.
\end{align*}
 % - \cfrac{(M+1)\sin(2M\theta_i) - (M+1)\sin(2\theta_i)-(M-1)\sin(2(M+1)\theta_i)}{M\sin(2\theta_i)}(\tilde{u}^{n_{k+1}}_{1,i}-\tilde{u}^{n_k}_{1,i}) \\

For \eqref{eq:case2_1-bsum}, we sum over $m=1,\cdots, M$ both sides of the equation and take $n=n_k$, this yields $u_{2,i}^{n_{k+1}}-u_{2,i}^{n_k}=$
\begin{equation} \label{eq:case2-1_sum}
  \cfrac{\sin(2M\theta_{i})}{\sin2\theta_{i}} (u_{2,i}^{n_k}-\tilde{u}_{2,i}^{n_k-1})+(\alpha_{i,M+1}-1)(u_{1,i}^{n_k} +u_{2,i}^{n_k})- \sum_{j=1}^{M}\mu_{i,j}(\tilde{u}^{n_{k+1}}_{1,i}-\tilde{u}^{n_k}_{1,i})
\end{equation}
where we use the fact that $\alpha_{i,1}=1$, $(\sum_{j=1}^{M}\alpha_{i,j})= \cfrac{\sin(2M\theta_{i})}{\sin2\theta_{i}}$, and
\begin{align*}
 \sum_{j=1}^{M}\mu_{i,j}&= \cfrac{(M+1)\sin(2M\theta_i) - (M+1)\sin(2\theta_i)-(M-1)\sin(2(M+1)\theta_i)}{M\sin(2\theta_i)}.
\end{align*}
%\end{proof}

 We note that $\tilde{u}_{2,i}^{n_k-1}$ in \eqref{eq:case2-1_sum} may be the solution computed in the previous coarse time interval using ``fine-fine" scheme, or from the values interpolated from the solution using ``coarse-fine" or ``coarse-coarse" scheme. We will discuss these two cases below.

\subsubsection[The values of u-tilde obtained by interpolation]{The values of $\tilde{u}_{2,i}^{n_k-1}$ obtained by interpolation}\label{sec:c-f-interp}

Suppose we do not have the value of $u_{2,i}^{n_k-1}$, from the scheme \eqref{eq:case2-1-u2_2}, we know
\begin{equation}\label{eq:case2-1_diff0}
    u_{2,i}^{n_k+1}-u_{2,i}^{n_k}=u_{2,i}^{n_k}-\tilde{u}_{2,i}^{n_k-1}-\tilde{\lambda}_{2,i}u_{2,i}^{n_k}-\tilde{\lambda}_{2,i} \tilde{u}_{1,i}^{n_k}.
\end{equation}
We would like to find the value $\tilde{u}_{2,i}^{n_k-1}$ using $u_{2,i}^{n_k}, u_{2,i}^{n_k+1}$ and $u_{2,i}^{n_{k-1}}$ such that 
\begin{equation}\label{eq:case2-1_diff}
    \cfrac{u_{2}^{n_k+1}-2u_{2}^{n_k}+\tilde{u}_{2}^{n_k-1}}{\Delta t^{2}}\approx\partial_{tt}u_2^{n_k}
\end{equation}

By Taylor expansion, we have
\begin{align*}
   u_{2,i}^{n_k+1} &=  u_{2,i}^{n_k} + \Delta t \partial_t u_{2,i}^{n_k} + \cfrac{\Delta t^2}{2} \partial_{tt} u_{2,i}^{n_k} +  \cfrac{\Delta t^3}{6} \partial_{ttt} u_{2,i}^{n_k} + {O}(\Delta t^4),\\
    u_{2,i}^{n_{k-1}} &=  u_{2,i}^{n_k} - M\Delta t \partial_t u_{2,i}^{n_k} + \cfrac{(M\Delta t)^2}{2} \partial_{tt} u_{2,i}^{n_k} - \cfrac{(M\Delta t)^3}{6} \partial_{ttt} u_{2,i}^{n_k} + {O}(\Delta t^4).
\end{align*}
Thus, 
\begin{align*}
 &u_{2,i}^{n_k+1}-2u_{2,i}^{n_k}+(u_{2,i}^{n_k}-M^{-1}(u_{2,i}^{n_k}-u_{2,i}^{n_k-1})) \\
 =&\cfrac{(1+M)\Delta t^{2}}{2}\partial_{tt}u_{2,i}^{n_k}+\cfrac{\Delta t^{3}(1-M^2)}{6}\partial_{ttt}u_{2,i}^{n_k}+O(\Delta t^{4})
\end{align*}
Therefore, we have 
\begin{equation}\label{eq:case2-1_diff2}
\partial_{tt}u_{2,i}^{n_k}\approx\cfrac{2(u_{2,i}^{n_k+1}-2u_{2,i}^{n_k}+(u_{2,i}^{n_k}-M^{-1}(u_{2,i}^{n_k}-u_{2,i}^{n_k-1})) )}{\Delta t^{2}(M+1)}
\end{equation}

From equations \eqref{eq:case2-1_diff} and \eqref{eq:case2-1_diff2}, we have 
\begin{equation*}
       \cfrac{u_{2,i}^{n_k+1}-2u_{2,i}^{n_k}+\tilde{u}_{2,i}^{n_k-1}}{\Delta t^{2}} = \cfrac{2(u_{2,i}^{n_k+1}-2u_{2,i}^{n_k}+(u_{2,i}^{n_k}-M^{-1}(u_{2,i}^{n_k}-u_{2,i}^{n_k-1})) )}{\Delta t^{2}(M+1)}, 
\end{equation*}
then
\begin{equation*}
    \tilde{u}_{2,i}^{n_k-1}=u_{2,i}^{n_k}-\cfrac{(M-1)(u_{2,i}^{n_k+1}-u_{2,i}^{n_k})}{M+1}-\cfrac{2(u_{2,i}^{n_k}-u_{2,i}^{n_{k-1}})}{M(M+1)}.
\end{equation*}

With the help \eqref{eq:case2-1_diff0}, we can eliminate $u_{2,i}^{n_k+1}$ in the above equation and solve for $\tilde{u}_{2,i}^{n_k-1}$, that is
\begin{equation}\label{eq:case2-1_diff3}
    \tilde{u}_{2,i}^{n_k-1}=u_{2,i}^{n_k}+\cfrac{\tilde{\lambda}_{2,i}(M-1)(u_{2,i}^{n_k}+ \tilde{u}_{1,i}^{n_k})}{2}-\cfrac{(u_{2,i}^{n_k}-u_{2,i}^{n_{k-1}})}{M}.
\end{equation}
We note that \eqref{eq:case2-1_diff3} is the interpolation formula we use in the numerical tests.

Plug \eqref{eq:case2-1_diff3} into \eqref{eq:case2-1_sum}, we get
\begin{align*}
u_{2,i}^{n_{k+1}}-u_{2,i}^{n_k}&=\beta_{i}(u_{2,i}^{n_k}-u_{2,i}^{n_{k-1}})-\xi_{i}(u_{2,i}^{n_k}+\tilde{u}_{1,i}^{n_k})- \sum_{j=1}^{M}\mu_{i,j}(\tilde{u}^{n_{k+1}}_{1,i}-\tilde{u}^{n_k}_{1,i}),
\end{align*}
where $\beta_{i}=\cfrac{\sin(2M\theta_{i})}{M\sin2\theta_{i}}$ and
\begin{align*}
\xi_{i}&=\cfrac{2(M-1)\sin^2\theta_{i}\sin(2M\theta_{i})}{\sin2\theta_{i}}-\cfrac{\sin(2(M+1)\theta_{i})-\sin(2M\theta_{i})}{\sin2\theta_{i}}+1.
\end{align*}
Thus,
\begin{align}
    (\cfrac{1+\beta_{i}}{2})(u_{2,i}^{n_{k+1}}-u_{2,i}^{n_k})=&(\cfrac{1+\beta_{i}}{2})(u_{2,i}^{n_k}-u_{2,i}^{n_{k-1}})-\cfrac{1-\beta_{i}}{2}(u_{2,i}^{n_{k+1}}-u_{2,i}^{n_{k-1}})\notag \\
    &-\xi_{i}(u_{2,i}^{n_k}+\tilde{u}_{1,i}^{n_k}) - \sum_{j=1}^{M}\mu_{i,j}(u_{1,i}^{n_{k+1}}-u_{1,i}^{n_k})\label{eq:case2-1_ite}.
\end{align}
Multiply both sides of equation \eqref{eq:case2-1_ite} by $\cfrac{\tilde{\lambda}_{2,i}M^2}{ \xi_{i}}(u_{2,i}^{n_{k+1}}-u_{2,i}^{n_{k-1}})$, we have 
\begin{align*}
    & (\cfrac{1+\beta_{i}}{2})\cfrac{\tilde{\lambda}_{2,i}M^2}{\xi_{i}}|u_{2,i}^{n_{k+1}}-u_{2,i}^{n_{k}} |^2=(\cfrac{1+\beta_{i}}{2})\cfrac{\tilde{\lambda}_{2,i}M^2}{ \xi_{i}}|u_{2,i}^{n_k}-u_{2,i}^{n_{k-1}} |^2 \\
    &-\cfrac{1-\beta_{i}}{2}\cfrac{\tilde{\lambda}_{2,i}M^2}{ \xi_{i}}|u_{2,i}^{n_{k+1}}-u_{2,i}^{n_{k-1}}|^2-\tilde{\lambda}_{2,i}M^2 (u_{2,i}^{n_k}+\tilde{u}_{1,i}^{n_k})(u_{2,i}^{n_{k+1}}-u_{2,i}^{n_{k-1}})\\
    &-\sum_{j=1}^{M}\mu_{i,j}\cfrac{\tilde{\lambda}_{2,i}M^2}{ \xi_{i}} (u_{1,i}^{n_{k+1}}-u_{1,i}^{n_k})(u_{2,i}^{n_{k+1}}-u_{2,i}^{n_{k-1}}),
\end{align*}
where we have $\cfrac{1-\beta_{i}}{2}\cfrac{\lambda_{2,i}}{ \xi_{i}}|u_{2,i}^{n_{k+1}}-u_{2,i}^{n_{k-1}}|^2\geq 0$, since $1-\beta_{i} = \cfrac{2}{3}(M^2-1)\theta_i^2 + O(\theta_i^4)$ by Taylor's expansion.

Sum over $i$ and re-organize terms, use $\tilde{\lambda}_{2,i}=\Delta t^2 \lambda_{2,i}$, we obtain
\begin{align}
    &\sum_i (\cfrac{1+\beta_{i}}{2})\cfrac{\tilde{\lambda}_{2,i}M^2}{\xi_{i}}|u_{2,i}^{n_{k+1}}-u_{2,i}^{n_{k}} |^2 + \Delta T^2 a(u_{2}^{n_{k+1}},u_{2}^{n_k}+ {u}_{1}^{n_k}) \label{case2-1_energy1}\\
    \leq &\sum_i (\cfrac{1+\beta_{i}}{2})\cfrac{\tilde{\lambda}_{2,i}M^2}{\xi_{i}}|u_{2,i}^{n_{k}}-u_{2,i}^{n_{k-1}} |^2
   +\Delta T^2a(u_{2}^{n_{k-1}},u_{2}^{n_k}+ {u}_{1}^{n_k})\notag\\
   &-\sum_{j=1}^{M}\mu_{i,j}\cfrac{\tilde{\lambda}_{2,i}M^2}{ \xi_{i}} (u_{1,i}^{n_{k+1}}-u_{1,i}^{n_k})(u_{2,i}^{n_{k+1}}-u_{2,i}^{n_{k-1}}).\notag
\end{align}

For the first equation in \eqref{eq:case2-1-u2}, we multiply $u_1^{n_{k+1}}-u_1^{n_{k-1}}$ both sides of the equation, we have
\begin{align}
   & \| u_{1}^{n_{k+1}}-u_{1}^{n_k}\|_b^2 + \cfrac{\Delta T^2}{2}  \left(\Vert u_{1}^{n_{k+1}}\Vert_a^2+\Vert u_{1}^{n_{k}}\Vert_a^2\right)+ \Delta T^2 
a(u_{2}^{n_k},u_{1}^{n_{k+1}}) \label{case2-1_energy2}\\
=&\| u_{1}^{n_k}-u_{1}^{n_{k-1}}\|_b^2 +\cfrac{\Delta T^2}{2}\left(\Vert u_{1}^{n_{k}}\Vert_a^2+\Vert u_{1}^{n_{k-1}}\Vert_a^2\right)+ \Delta T^2 a(u_{2}^{n_k},u_{1}^{n_{k-1}})\notag
\end{align}
Sum \eqref{case2-1_energy1} and \eqref{case2-1_energy2}, we get 
\begin{align*}\label{case2-1_energy3}
    & \| u_{1}^{n_{k+1}}-u_{1}^{n_k}\|_b^2 +\sum_i (\cfrac{1+\beta_{i}}{2})\cfrac{\tilde{\lambda}_{2,i}M^2}{\xi_{i}}|u_{2,i}^{n_{k+1}}-u_{2,i}^{n_{k}} |^2  +\cfrac{\Delta T^2}{2} \Vert u_{1}^{n_{k+1}}\Vert_a^2\\
    &+\cfrac{\Delta T^2}{2}  \Vert u_{1}^{n_{k}}\Vert_a^2+\Delta T^2 \left( a(u_{2}^{n_{k+1}},{u}_{1}^{n_k})+ a(u_{2}^{n_{k+1}},u_{2}^{n_k}) +a(u_{1}^{n_{k+1}}, u_{2}^{n_k})  \right)\\
    \leq &\| u_{1}^{n_k}-u_{1}^{n_{k-1}}\|_b^2+\sum_i (\cfrac{1+\beta_{i}}{2})\cfrac{\tilde{\lambda}_{2,i}M^2}{\xi_{i}}|u_{2,i}^{n_{k}}-u_{2,i}^{n_{k-1}} |^2  +\cfrac{\Delta T^2}{2}\Vert u_{1}^{n_{k}}\Vert_a^2\\
    &+\cfrac{\Delta T^2}{2}\Vert u_{1}^{n_{k-1}}\Vert_a^2+\Delta T^2 \left(a(u_{2}^{n_k},u_{1}^{n_{k-1}})+ a(u_{2}^{n_k},u_{2}^{n_{k-1}})+ a({u}_{1}^{n_k},u_{2}^{n_{k-1}}) \right)\\
    &-\sum_{j=1}^{M}\mu_{i,j}\cfrac{\tilde{\lambda}_{2,i}M^2}{ \xi_{i}} (u_{1,i}^{n_{k+1}}-u_{1,i}^{n_k})(u_{2,i}^{n_{k+1}}-u_{2,i}^{n_{k-1}}).
\end{align*}

 Since $\xi_{i}=1-\cos(2M\theta_{i})+\cfrac{M\sin\theta_{i}\sin(2M\theta_{i})}{\cos\theta_{i}}$, then
\begin{align*}
\cfrac{1+\beta_{i}}{2}\cfrac{M^{2}\tilde{\lambda}_{i}}{\xi_{i}}=1+O(M^2\theta_{i}^{2}), \;\; 
\cfrac{\sum_{j=1}^{M}\mu_{i,j}}{ \xi_{i}}  = \cfrac{M^2-1}{6M^2} + O(\theta_i^2).
%\sum_{j=1}^{M}\mu_{i,j}\cfrac{\tilde{\lambda}_{2,i}M^2}{ \xi_{i}}  = \cfrac{2}{3}(M^2-1)\theta_i^2 + O(\theta_i^4).
\end{align*}
Thus, we have
\begin{align*}
    & \| u_{1}^{n_{k+1}}-u_{1}^{n_k}\|_b^2 + \| u_{2}^{n_{k+1}}-u_{2}^{n_k}\|_b^2 +\cfrac{\Delta T^2}{2}  \left(\Vert u_{1}^{n_{k+1}}\Vert_a^2+\Vert u_{1}^{n_{k}}\Vert_a^2\right)\\
    &+\Delta T^2 \left( a(u_{2}^{n_{k+1}},{u}_{1}^{n_k})+ a(u_{2}^{n_{k+1}},u_{2}^{n_k}) +a(u_{1}^{n_{k+1}}, u_{2}^{n_k})  \right)\\
    \leq &\| u_{1}^{n_k}-u_{1}^{n_{k-1}}\|_b^2+ \| u_{1}^{n_k}-u_{1}^{n_{k-1}}\|_b^2  +\cfrac{\Delta T^2}{2}\left(\Vert u_{1}^{n_{k}}\Vert_a^2+\Vert u_{1}^{n_{k-1}}\Vert_a^2\right)\\
    &+\Delta T^2 \left(a(u_{2}^{n_k},u_{1}^{n_{k-1}})+ a(u_{2}^{n_k},u_{2}^{n_{k-1}})+ a({u}_{1}^{n_k},u_{2}^{n_{k-1}})\right) \\
    & + C \max\{M^2\theta_i^2\} (\| u_{2}^{n_{k+1}}-u_{2}^{n_k}\|_b^2+\|u_{1,i}^{n_{k+1}}-u_{1,i}^{n_k}\|_b^2 +\|u_{2,i}^{n_{k+1}}-u_{2,i}^{n_{k-1}}\|_b^2).
\end{align*}

\subsubsection[The values of u-tilde are known form previous time step]{The values of $u_{2,i}^{n_k-1}$ are known form previous time step}
From \eqref{eq:case2-1-u2_2}, we can also have 
\begin{equation*}
    \left(\begin{array}{c}
u_{2,i}^{n-1}\\
u_{2,i}^{n}
\end{array}\right)=\left(\begin{array}{cc}
2-\tilde{\lambda}_{2,i} & -1\\
1 & 0
\end{array}\right)\left(\begin{array}{c}
u_{2,i}^{n}\\
u_{2,i}^{n+1}
\end{array}\right) -\left(\begin{array}{c}
\tilde{\lambda}_{2,i}  \tilde{u}_{1,i}^{n_k}\\
0
\end{array}\right).
\end{equation*}
With a similar computation as before, we have 
\begin{equation*}   
u_{2,i}^{n-m}-u_{2,i}^{n-m+1}= \alpha_{i,m+1}u_{2,i}^{n} -\alpha_{i,m}u_{2,i}^{n+1}+\gamma_{i,m} \tilde{u}_{1,i}^{n_k},
\end{equation*}

Sum over $m=1,\cdots, M$ both sides of the equation and take $n=n_k$ yields
\begin{align}\label{eq:case2-2_sum}
   u_{2,i}^{n_{k-1}}-u_{2,i}^{n_k}=\cfrac{\sin(2M\theta_{i})}{\sin2\theta_{i}}(u_{2,i}^{n_k}-u_{2,i}^{n_k+1})+(\alpha_{i,M+1}-1)(u_{2,i}^{n_k}+u_{1,i}^{n_k}). 
\end{align}

Add $u_{2,i}^{n_{k+1}}-u_{2,i}^{n_k}=\cfrac{\sin(2M\theta_{i})}{\sin2\theta_{i}} (u_{2,i}^{n_k}-u_{2,i}^{n_k-1})+(\alpha_{i,M+1}-1)(u_{1,i}^{n_k} +u_{2,i}^{n_k})$ with \eqref{eq:case2-2_sum}, we have 
\begin{align*}
&u_{2,i}^{n_{k+1}}-2u_{2,i}^{n_k}+u_{2,i}^{n_{k-1}} \\=& - M\beta_i (u_{2,i}^{n_k+1}-2u_{2,i}^{n_k}-u_{2,i}^{n_k-1}) + 2(\alpha_{i,M+1}-1
    )(u_{2,i}^{n_k}+\tilde{u}_{1,i}^{n_k})\\
    = &( M\beta_i \lambda_{2,i}+2(\alpha_{i,M+1}-1)(u_{2,i}^{n_k}+\tilde{u}_{1,i}^{n_k}) ): = -\eta_i(u_{2,i}^{n_k}+\tilde{u}_{1,i}^{n_k}),
\end{align*}
where we note that 
\begin{align*}
\eta_i &= -M\beta_i \lambda_{2,i}-2(\alpha_{i,M+1}-1)\\
&=2-\cfrac{4\sin2M\theta_i \sin^2\theta_i+2 (\sin((2M+1)\theta_{i})-\sin(2M\theta_i))}{\sin(2M\theta_{i})}\\
 & =4\sin^2(M\theta_i) \approx 4\Delta T^{2}\lambda_{2,i}-O(M^{4}\theta_i^{4}).
\end{align*}

Multiply both sides of $u_{2,i}^{n_{k+1}}-2u_{2,i}^{n_k}+u_{2,i}^{n_{k-1}} + \eta_i(u_{2,i}^{n_k}+\tilde{u}_{1,i}^{n_k}) = 0$ by $u_{2,i}^{n_{k+1}}-u_{2,i}^{n_{k-1}}$, sum over $i$, we have 
\begin{align*}
&\|u_{2,i}^{n_{k+1}}-u_{2,i}^{n_k}\|_b^2+ \Delta T^{2}(a(u_{2,i}^{n_k},u_{2,i}^{n_{k+1}})+\Delta T^{2}a(\tilde{u}_{1,i}^{n_k}, u_{2,i}^{n_{k+1}}) )\\
=& \|u_{2,i}^{n_k}-u_{2,i}^{n_{k-1}}\|_b^2 +\Delta T^{2}(a(u_{2,i}^{n_k},u_{2,i}^{n_{k-1}})+\Delta T^{2}a(\tilde{u}_{1,i}^{n_k},u_{2,i}^{n_{k-1}}) )\\
&+C\max\{M^4\theta_i^4\} (u_{2,i}^{n_k}+\tilde{u}_{1,i}^{n_k})(u_{2,i}^{n_{k+1}}-u_{2,i}^{n_{k-1}}).
\end{align*}
 Sum with \eqref{case2-1_energy2}, we get
 \begin{align*}
&\| u_{1}^{n_{k+1}}-u_{1}^{n_k}\|_b^2 + \cfrac{\Delta T^2}{2}  \left(\Vert u_{1}^{n_{k+1}}\Vert_a^2+\Vert u_{1}^{n_{k}}\Vert_a^2\right)+\|u_{2,i}^{n_{k+1}}-u_{2,i}^{n_k}\|_b^2\\
&+ \Delta T^{2}(a(u_{2,i}^{n_k},u_{2,i}^{n_{k+1}})+\Delta T^{2}a(\tilde{u}_{1,i}^{n_k}, u_{2,i}^{n_{k+1}}) )+ \Delta T^2 
a(u_{2}^{n_k},u_{1}^{n_{k+1}})\\
=& \| u_{1}^{n_k}-u_{1}^{n_{k-1}}\|_b^2 +\cfrac{\Delta T^2}{2}\left(\Vert u_{1}^{n_{k}}\Vert_a^2+\Vert u_{1}^{n_{k-1}}\Vert_a^2\right)+\|u_{2,i}^{n_k}-u_{2,i}^{n_{k-1}}\|_b^2 \\
&+\Delta T^{2}(a(u_{2,i}^{n_k},u_{2,i}^{n_{k-1}})+\Delta T^{2}a(\tilde{u}_{1,i}^{n_k},u_{2,i}^{n_{k-1}}) )+ \Delta T^2 a(u_{2}^{n_k},u_{1}^{n_{k-1}})\\
&+C\max\{M^4\theta_i^4\} (\|u_{2,i}^{n_k}\|^2_b+\|{u}_{1,i}^{n_k}\|^2_b+2\|u_{2,i}^{n_{k+1}}-u_{2,i}^{n_{k-1}}\|^2_b).
\end{align*}
By the definition of discrete energy $\tilde{E}^{n_{k+\frac{1}{2}}}$, we have $\tilde{E}^{n_{k+\frac{1}{2}}}\leq \tilde{E}^{n_{k-\frac{1}{2}}}$. Then follow a similar proof as in Lemma \ref{l1}, we have the stability result for Case 2.

\subsection{Adaptive multirate time splitting algorithm}
\label{section: Adaptive multirate time splitting algorithm}
In this section, we will propose a residual-based adaptive multirate algorithm \cite{adaptivemultirate} to select a suitable time step size for the decoupled splitting system. The proposed approach involves the construction of an error indicator derived from residuals. The indicators provide an estimation of discretization errors when coarse time steps are employed for both equations in \eqref{D1}. Based on this estimation, adaptive time step refinement can be selectively applied to components with large errors.

We first derive the error indicators. Let $U(t)$ be the piecewise quadratic function with $U(T^n)=u_H^n=u_{H,1}^n+u_{H,2}^n$ such that on $(T_{n-1},T_{n+1}]$
\begin{align}\label{G1}
u(t)=u^n+\frac{u^{n+1}-u^{n-1}}{2\Delta T}(t-T_n)+\frac{u^{n+1}-2u^n+u^{n-1}}{2\Delta T^2}(t-T_n)^2,
\end{align}
and let $F(t)$ be the piecewise constant such that on $(T_{n-1},T_{n+1}]$, it is equal to the $L^2-$projection of $f^{n+1}$ onto multiscale space $V_H$, i.e.,
\begin{align}\label{G2}
(F(t),v)=(f^{n+1},v),\ \forall v\in V_H.
\end{align}
Define $X_n:=L^2((T_{n-1},T_{n+1}];H_0^1(\Omega))$ equipped with the inner product and norm as
\begin{align*}
(u,v)_{X_n}&=\int_{T_{n-1}}^{T_{n+1}}\left(\int_{\Omega}uv+\int_{\Omega}\kappa\nabla u\cdot\nabla v\right)dt,\\ 
\Vert v\Vert_{X_n}&=\left(\int_{T_{n-1}}^{T_{n+1}}(\Vert v\Vert^2+\Vert v\Vert_a^2)dt\right)^{\frac{1}{2}}.
\end{align*}
Define a constant $\gamma$ depending on the multiscale space $V_{H,1}$ and $V_{H,2}$ as
\begin{align*}
\gamma=\sup\limits_{\substack{v_1\in L^2((T_{n-1},T_{n+1}];V_{H,1})\\ v_2\in L^2((T_{n-1},T_{n+1}];V_{H,1})}}\frac{(v_1,v_2)}{\Vert v_1\Vert_{X_n}\Vert v_2\Vert_{X_n}}<1.
\end{align*}
Let the integral of the residual over the time interval $(T_{n-1},T_{n+1}]$ be
\begin{align*}
\int_{T_{n-1}}^{T_{n+1}}(R(u(t)),v)dt=\int_{T_{n-1}}^{T_{n+1}}(f(t),v)-b(u''(t),v)-a(u(t),v)dt
\end{align*}
and define the dual norm of the residual
\begin{align*}
\mathcal{R}_v=\sup\limits_{\substack{v\in L^2((T_{n-1},T_{n+1}],V_H)\\ \Vert v\Vert_{X_n}=1}}\int_{T_{n-1}}^{T_{n+1}}(R(u(t)),v)dt.
\end{align*}
Then we have following estimates.

\begin{thm}\label{thm:adaptive}
Define the following error indicators:
\begin{align}\label{G3}
\begin{aligned}
\eta_1^n=&(\sqrt{\frac{\Delta T}{2}}+\sqrt{\frac{\Delta T}{6}}+\sqrt{\frac{\Delta T}{10}}) (\Vert u_{H,1}^{n+1}-u_{H,1}^{n}\Vert_a+\Vert u_{H,1}^{n}-u_{H,1}^{n-1}\Vert_a)\\
&+(\sqrt{\frac{\Delta T}{6}}+\sqrt{\frac{\Delta T}{10}}) (\Vert u_{H,2}^{n+1}-u_{H,2}^{n}\Vert_a+\Vert u_{H,2}^{n}-u_{H,2}^{n-1}\Vert_a),\\
\eta_2^n=&(\sqrt{\frac{\Delta T}{6}}+\sqrt{\frac{\Delta T}{10}}) (\Vert u_{H,1}^{n+1}-u_{H,1}^{n}\Vert_a+\Vert u_{H,1}^{n}-u_{H,1}^{n-1}\Vert_a\\
&+\Vert u_{H,2}^{n+1}-u_{H,2}^{n}\Vert_a+\Vert u_{H,2}^{n}-u_{H,2}^{n-1}\Vert_a).
\end{aligned}
\end{align}
Then there exist constant $C_3$ and $C_4$ such that
\begin{align*}
\mathcal{R}_v\leq C_3(1-\gamma)^{-\frac{1}{2}}((\eta_1^n)^2+(\eta_2^n)^2)^{\frac{1}{2}}+C_4\Vert f(t)-F(t)\Vert_{L^2((T_{n-1},T_{n+1}];L^2(\Omega))}.
\end{align*}
\end{thm}
The proof is straightforward, and we putsketch of proof for the theorem in the Appendix.

\section{Numerical experiments}
\label{section: Numerical experiments}
This section details the numerical tests conducted to assess the effectiveness and efficiency of our numerical scheme. In the following examples, we consider $\Omega=[0,1]^2$, we apply zero Neumann boundary conditions, and initial conditions $u(\cdot,0)=0, u_t(\cdot,0)=0$.
The true solutions are approximated using FEM on fine grid (grid size $h$) in space, and implicit Euler scheme in time (step size $\delta t$). To examine only the error in temporal discretization, our reference solutions are computed using a finer time step size $\delta t$ with the implicit Euler scheme, and the spatial discretizations were as discussed in Section \ref{section: Construction of two multiscale subspaces}. In all examples, we let $\delta_1$ and $\delta_2$ denotes the thresholds.

\subsection{Example 1: high-contrast media and point sourece}
\label{section: Numerical experiment 1}
In our first numerical test, we take following parameters in the time discretizations
\begin{align*}
T=0.3,\ \Delta T=1\times10^{-3},\ \Delta t=1\times10^{-4},\ \delta t=5\times10^{-5}.
\end{align*}
We consider the following smooth source term
\begin{align*}
f(x,y,t)=\sin{(f_0 t)}f_{init},
\end{align*}
where $f_0$ is the frequency. In our experiments, we will consider the different values for $f_0$ to illustrate the feasibility and effectiveness of our proposed algorithm. The medium parameter $\kappa$ and the source $f_{init}$ are presented in Figure \ref{ex1_kappa_f}. As we see that the permeability field is high contrast with contrast ratio $10^3$, $f_{init}$ is a point source in the middle of the computational domain.
\begin{figure}[H]
\centering
\includegraphics[width=0.8\linewidth]{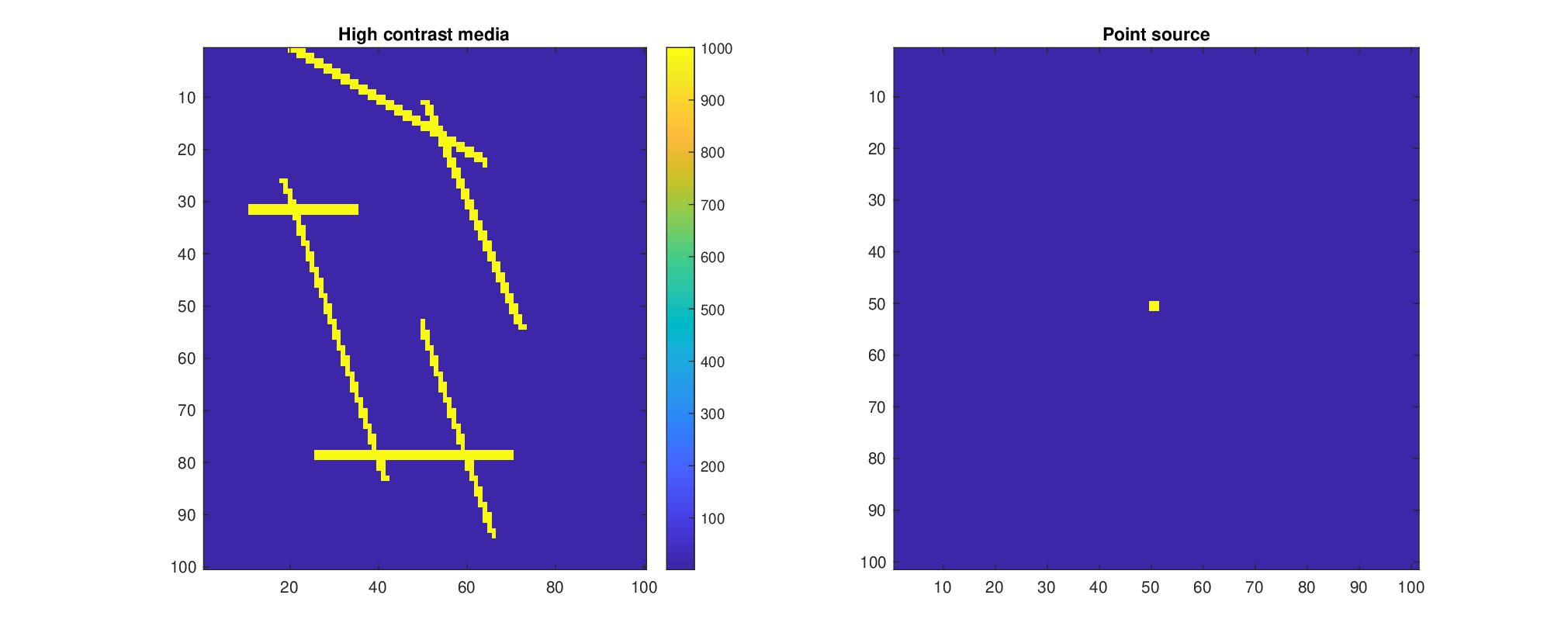}
\caption{Example 1. \textbf{Left:} $\kappa$; \textbf{Right:} $f_{init}$.}\label{ex1_kappa_f}
\end{figure}
\subsubsection[The first case]{The first case: $f_0=5$}
\label{section:The first case: $f_0=5$}
Here, the frequency in the source term is set to be $f_0=5$. In this case, we take $H=\frac{1}{10},\ h=\frac{1}{100}$. In Figure \ref{ex1_5_errors}, we show the relative $L^2$ errors and relative energy errors between the reference solution and splitting scheme using uniform coarse time steps, using fine time steps and using adaptive time refinement, correspondingly. The error indicators in the right of Figure \ref{ex1_5_errors} illustrate that our proposed numerical algorithm selects coarse time step for $u_1$ and fine time stpe size for $u_2$, and the results for the multirate step size coincide with the results when one uses both fine step sizes for the two equations in the splitting scheme. This indicates that our multirate method can save computational cost while maintaining accuracy. The solution profile at the final time are presented in Figure \ref{ex1_5_solT}. We can observe that the solution obtained from our method matches the approximated true solution computed using FEM (with grid size $h$) and implicit Euler(with time step size$\delta t$). The degrees of freedom (DOF) for FEM is $10201$, and the DOF for multiscale method is $826$ (where $\text{dim}(V_{H,1}) = 26$ and we take $l_i=8$ in $V_{aux,2}$). %At first, the algorithm refines time step for the first equation and chooses coarse time step for the second equation, then it eventually stabilizes to the fine-fine case at the latter part of the simulation. The number of refined coarse blocks for both equations is around $195/400$. This demonstrate that our proposed algorithm can achieve accuracy comparable to that of fine-fine solutions while significantly reducing computational cost.

\begin{figure}[H]
\centering
\includegraphics[width=0.3\linewidth]{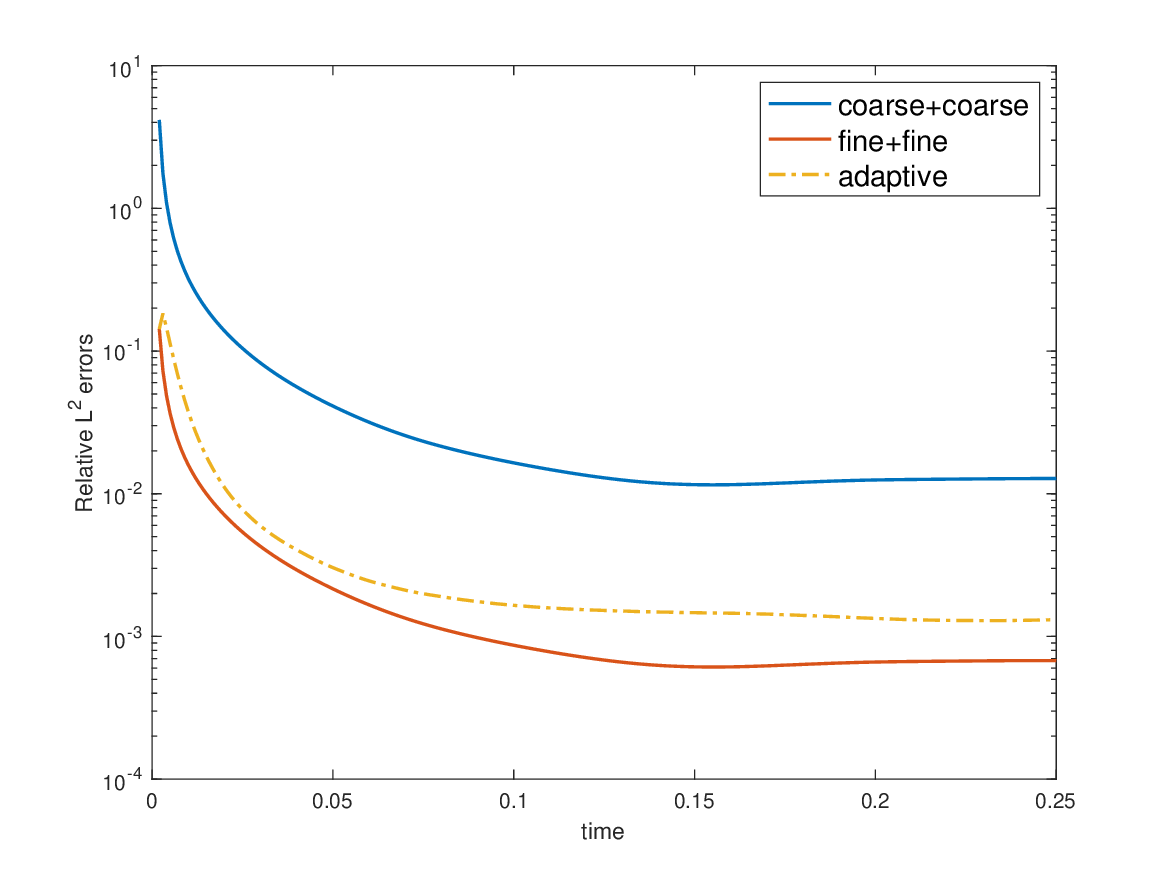}
\includegraphics[width=0.3\linewidth]{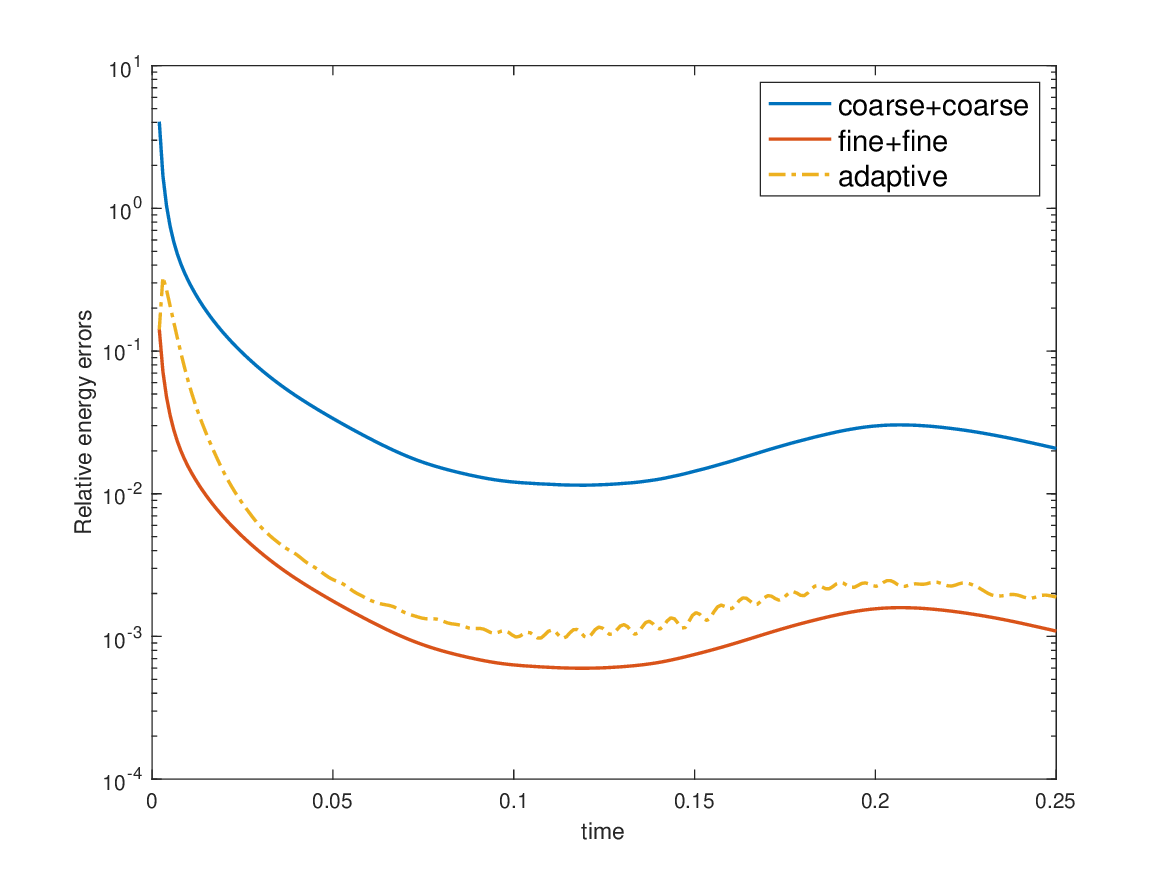}
\includegraphics[width=0.3\linewidth]{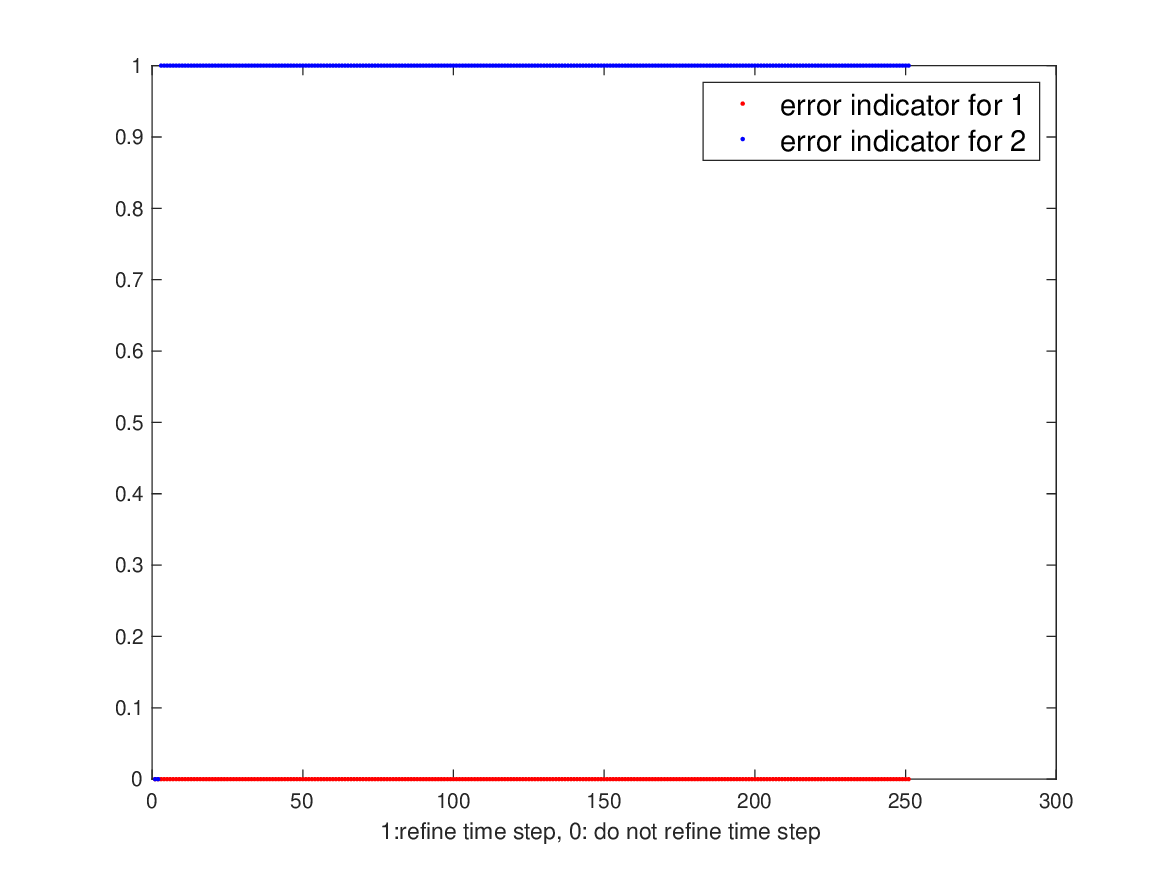}
\caption{Example 1. \textbf{Left:} relative $L_2$ error, \textbf{Middle:} relative energy error. \textbf{Right:} refinement history.}\label{ex1_5_errors}
\end{figure}

\begin{figure}[H]
\centering
\includegraphics[width=1.0\linewidth]{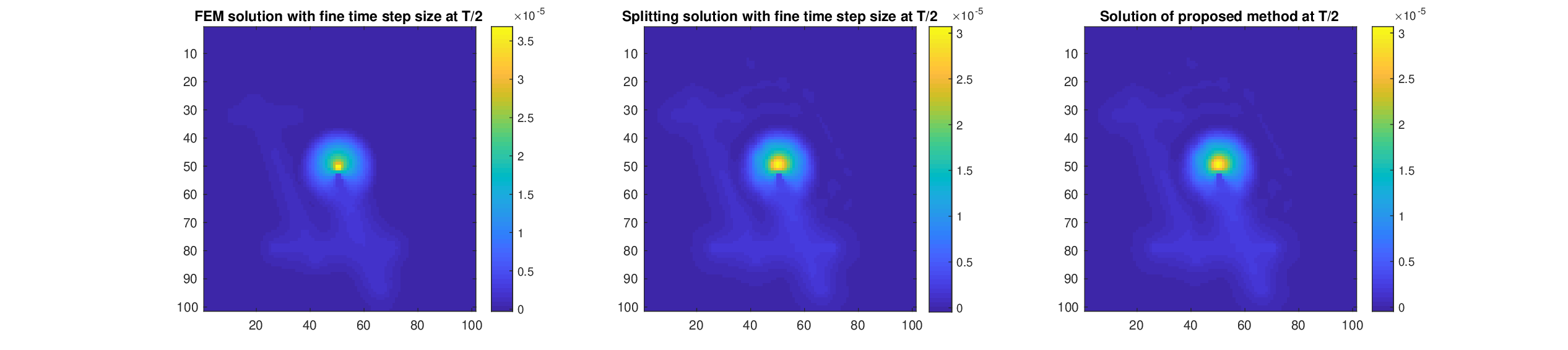}
\includegraphics[width=1.0\linewidth]{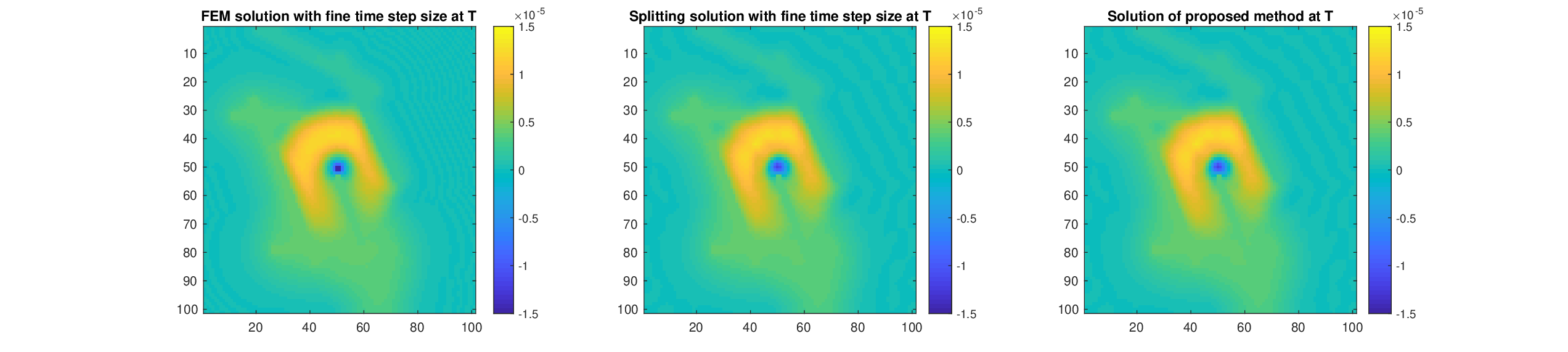}
\caption{Example 1. Left: FEM ($h$) and implicit Euler( $\delta t$). Middle: multiscale ($H$) splitting with fine step size $\delta t$), Right: multiscale ($H$) splitting with adaptive step size. Top: solutions at $T/2$; bottom: solutions at $T$.}\label{ex1_5_solT}
\end{figure}

\subsubsection[The second case]{The second case: $f_0=40$}
\label{section:The second case: $f_0=40$}
In this subsection, we take the frequency $f_0=40$, and the mesh size $H=\frac{1}{20},\ h=\frac{1}{200}$ correspondingly. Similarly, the relative errors  between the reference solution and splitting scheme using uniform coarse time steps, using fine time steps and using adaptive time refinement are presented in Figure \ref{ex1_40_errors}, correspondingly. The error indicators are the same as in the right of Figure \ref{ex1_40_errors}, so we omit it here. We see that the errors for the multirate method are still very close to the errors when we use fine-fine step sizes for the two equations in the splitting scheme.  The solution profile at the time $T/2$ and final time $T$ are presented in Figure \ref{ex1_40_solT}, which demonstrate a good agreement between approximate true solutions and solutions obtained from our method. The degrees of freedom (DOF) for FEM is $40401$, and multiscale method is $3252$ (where $\text{dim}(V_{H,1}) = 52$ and we take $l_i=8$ in $V_{aux,2}$). The mass lumping errors are shown in Figure \ref{ex1_ml_err}.

\begin{figure}[H]
\centering
\includegraphics[width=0.3\linewidth]{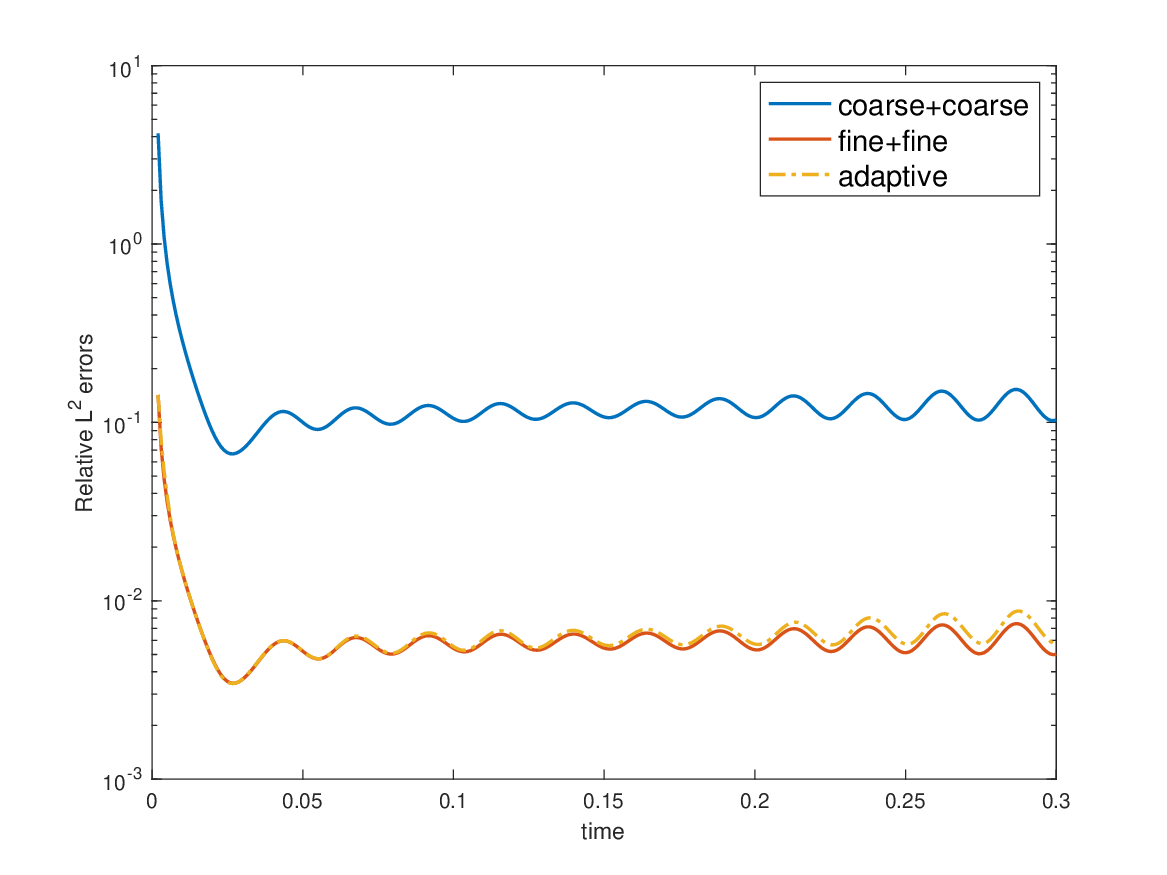}
\includegraphics[width=0.3\linewidth]{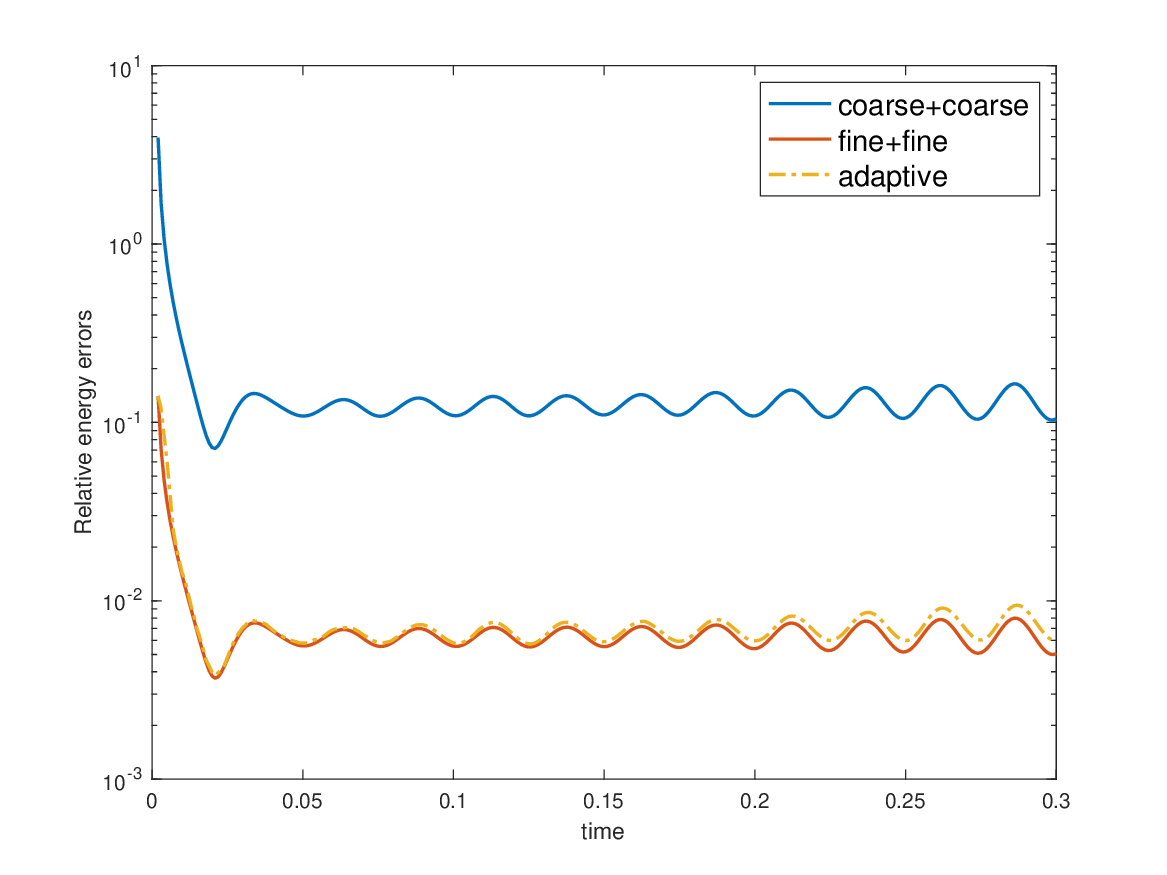}
\caption{Example 1. \textbf{Left:} relative $L_2$ error, \textbf{Right:} relative energy error.}\label{ex1_40_errors}
\end{figure}

\begin{figure}[H]
\centering
\includegraphics[width=1.0\linewidth]{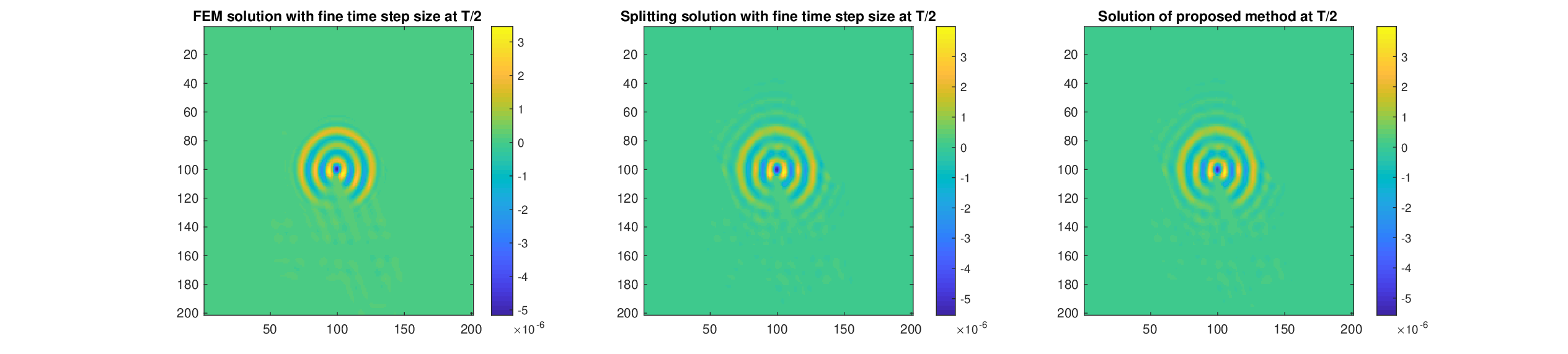}
\includegraphics[width=1.0\linewidth]{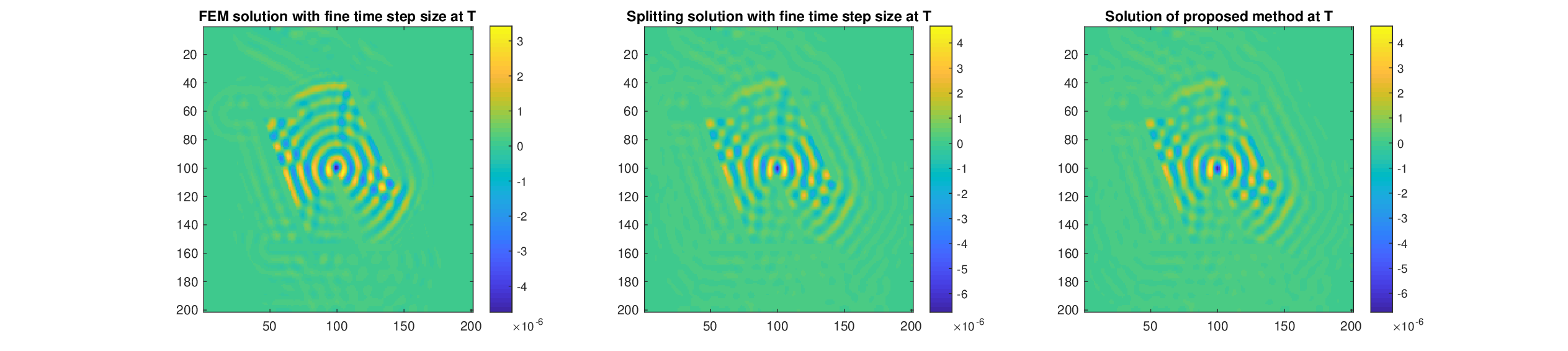}
\caption{Example 1. Left: FEM ($h$) and implicit Euler( $\delta t$). Middle: multiscale ($H$) splitting with fine step size $\delta t$), Right: multiscale ($H$) splitting with adaptive step size. Top row: solution at middle time. Bottom row: solution at final time. }\label{ex1_40_solT}
\end{figure}

\begin{figure}[H]
\centering
\includegraphics[width=0.3\linewidth]{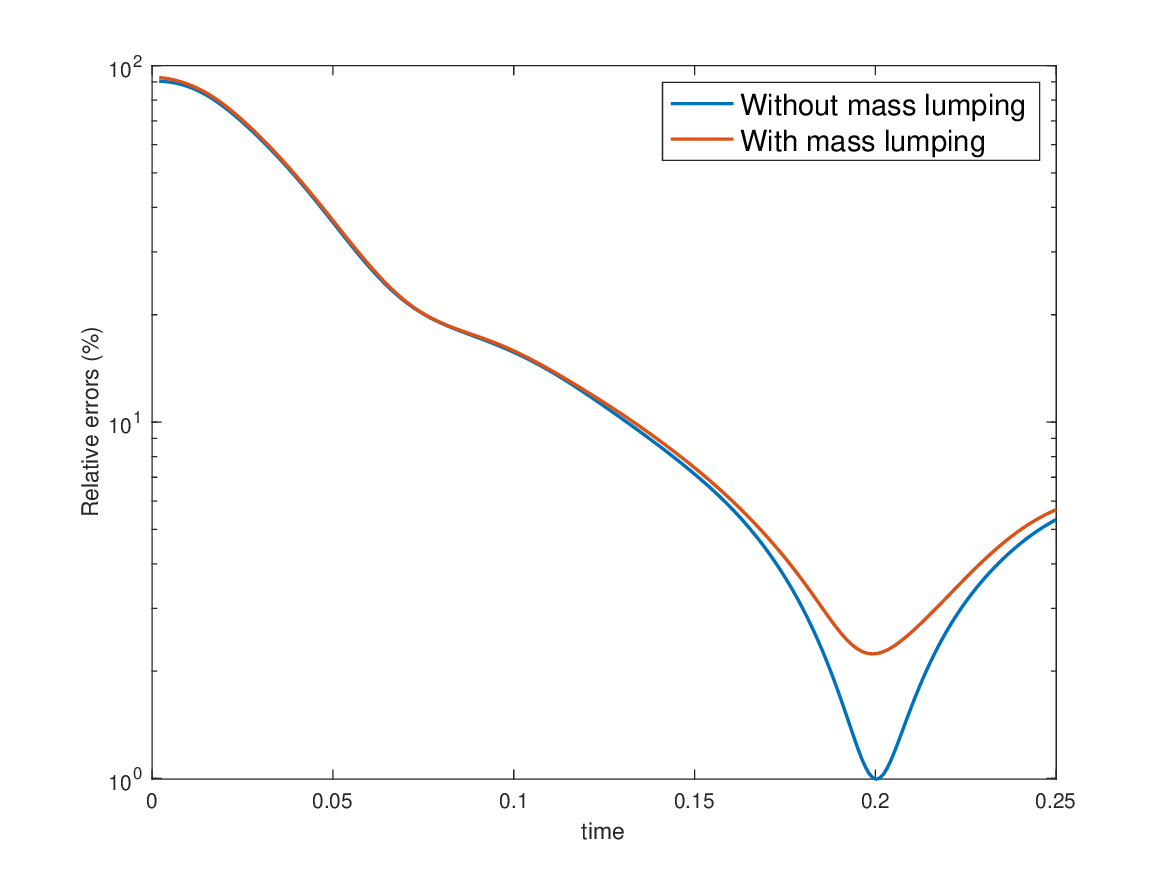}
\includegraphics[width=0.3\linewidth]{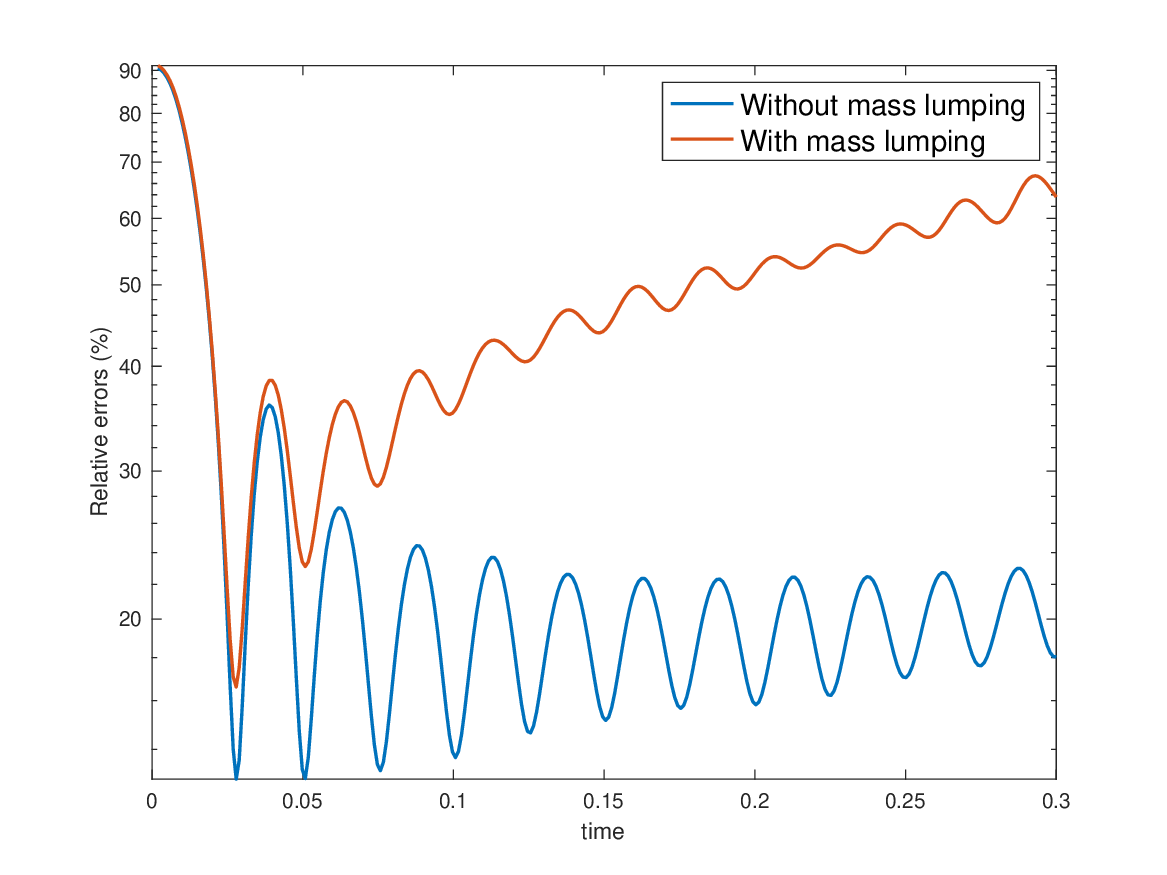}
\caption{Example 1.  The errors between the FEM and multiscale solutions with or without mass lumping. Left: $f_0=5\pi$; Right: $f_0=40\pi$}\label{ex1_ml_err}
\end{figure}

\subsection{Example 2: Modified Marmousi model}
\label{section: Numerical experiment 2}

In our second example, we take $\ H=0.1,\ h=0.01$ and 
\begin{align*}
t\in [0.05,0.20],\ \Delta T=5\times10^{-4},\ \Delta t=5\times10^{-5},\ \delta t=2.5\times10^{-5}.
\end{align*}
We consider the following smooth source term
\begin{align*}
f(x,y,t)&=f_{xy}\exp(-\pi^2 f_0^2(t-\frac{2}{f_0})^2)(1-2\pi^2 f_0^2(t-\frac{2}{f_0})^2),\\
f_{xy} &= 5000 \exp(-\frac{(x-\frac{1}{2})^2+(y-\frac{1}{2})^2}{\sigma^2} )
\end{align*}
where $f_0=15$, $\sigma^2 = 2\times10^{-4}$. The medium parameter $\kappa$ is a modified Marmousi moddel, and the source $f_{xy}$ are presented in Figure \ref{ex2_kappa_f}.

\begin{figure}[H]
\centering
\includegraphics[width=0.6\linewidth]{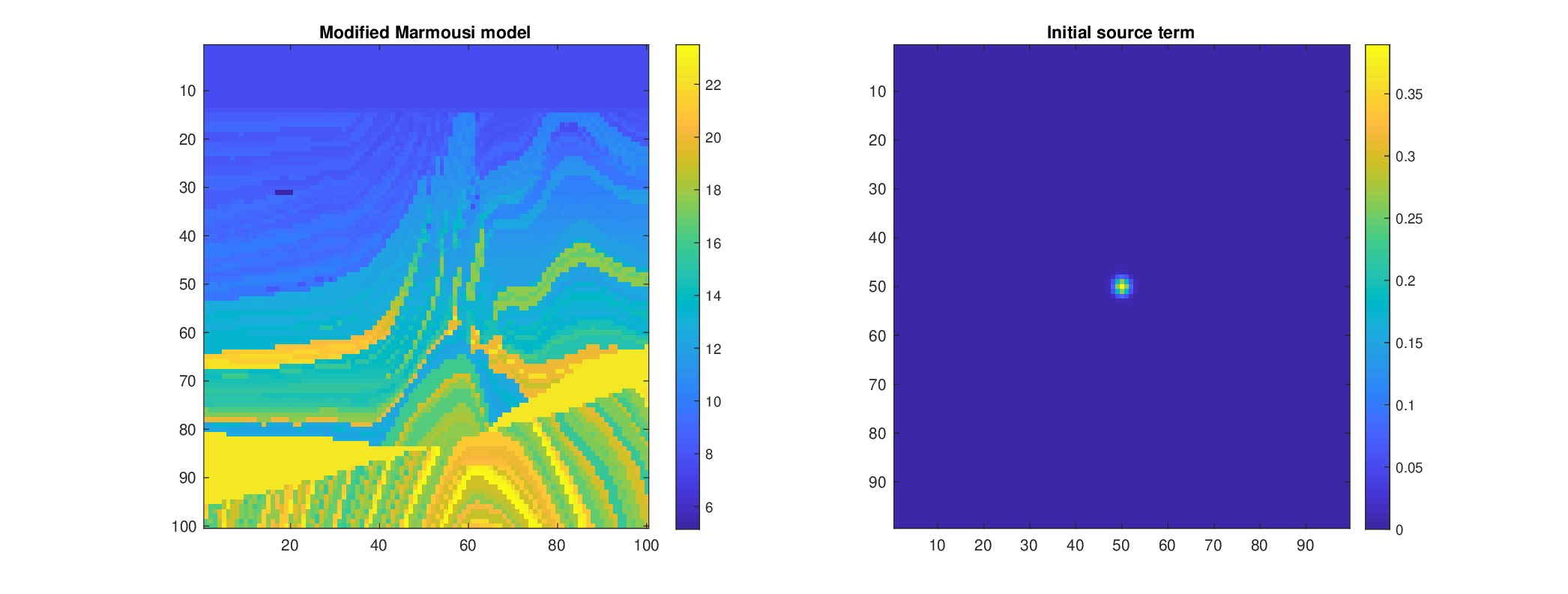}
\includegraphics[width=0.3\linewidth]{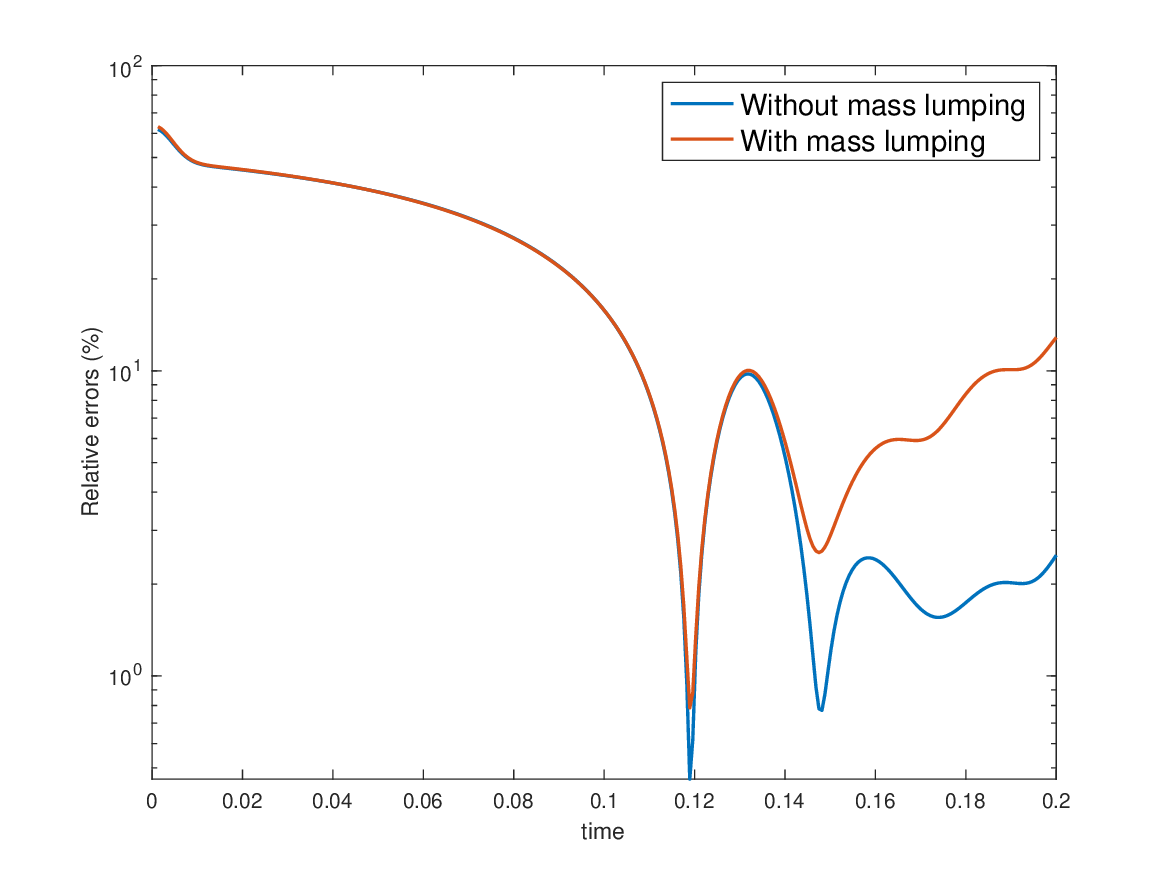}
\caption{Example 2. Left: $\kappa$; Middle: $f_{init}$. Right: the errors between the FEM and multiscale solutions with or without mass lumping.}\label{ex2_kappa_f}
\end{figure}

 The solution profiles at the time $T=0.15, 0.2, 0.25$ are presented in Figure \ref{ex2_solT}, which demonstrate a good agreement between approximate true solutions and solutions obtained from our method. Our implementation is conducted on MATLAB. The computational time for the solution using finite element with implicit time discretization is 46.09 seconds, while the solution using NLMC partially splitting sheme with mass lumping is 13.46 seconds, and the adaptive scheme proposed in this work takes  4.29 seconds. The degrees of freedom (DOF) for FEM is $10201$, and multiscale method is $839$ where $\text{dim}(V_{aux,1}) = 39$ we take $l_i=8$ in $V_{aux,2}$. We compare errors between the splitting scheme using uniform coarse time steps, using fine time steps and using adaptive time refinement with reference solutions, and the errors are shown in the left and middle of Figure \ref{ex2_errors}.  The refinement history are shown in the right of \ref{ex2_errors}. This demonstrate that our proposed algorithm can achieve accuracy comparable to that of fine-fine solutions while reducing computational cost.%At first, the algorithm take coarse time step for the second equation, and when the source term fluctuates significantly, it begins to refine the second equation. It eventually stabilizes to the fine-fine case at the latter part of the simulation. The number of refined coarse blocks for both equations is around $232/400$. 

\begin{figure}[H]
\centering
\includegraphics[width=1.0\linewidth]{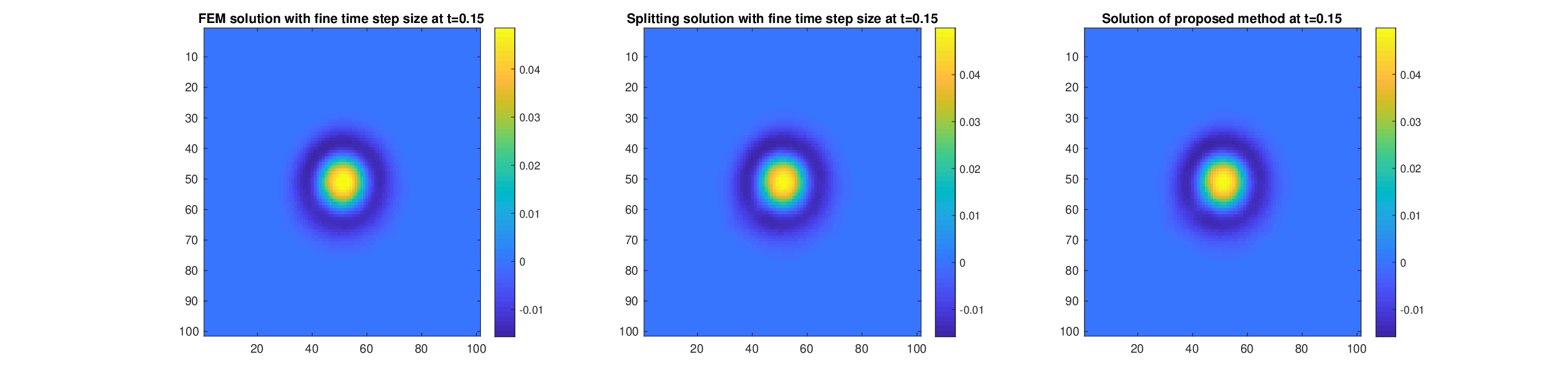}
\includegraphics[width=1.0\linewidth]{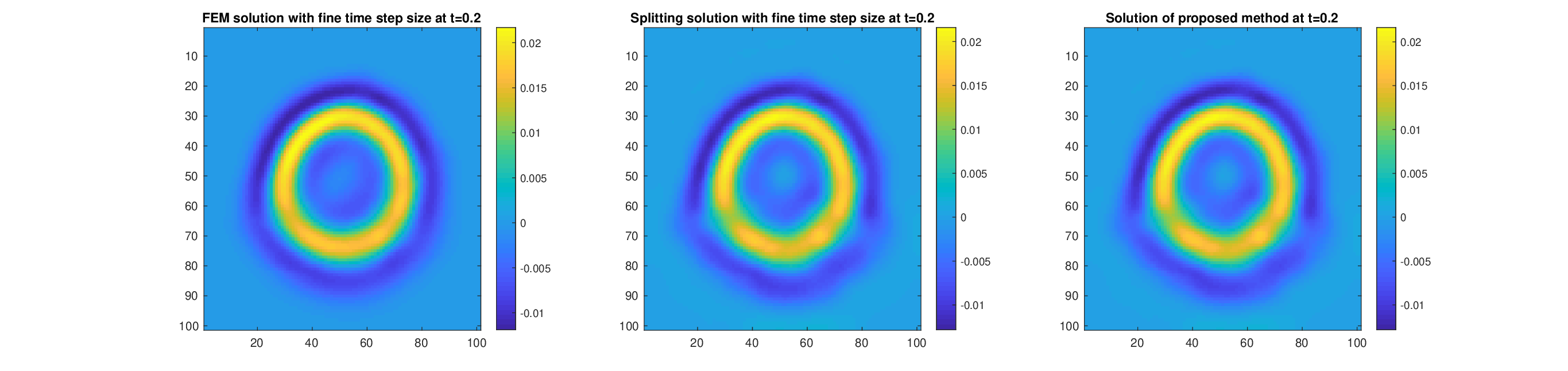}
\caption{Example 2. Solution at final time. Left: FEM ($h$) and implicit Euler( $\delta t$). Middle: multiscale ($H$) splitting with fine step size $\delta t$), Right: multiscale ($H$) splitting with adaptive step size. Top row: $T=0.15$, second row: $T=0.2$.}\label{ex2_solT}
\end{figure}

\begin{figure}[H]
\centering
\includegraphics[width=0.3\linewidth]{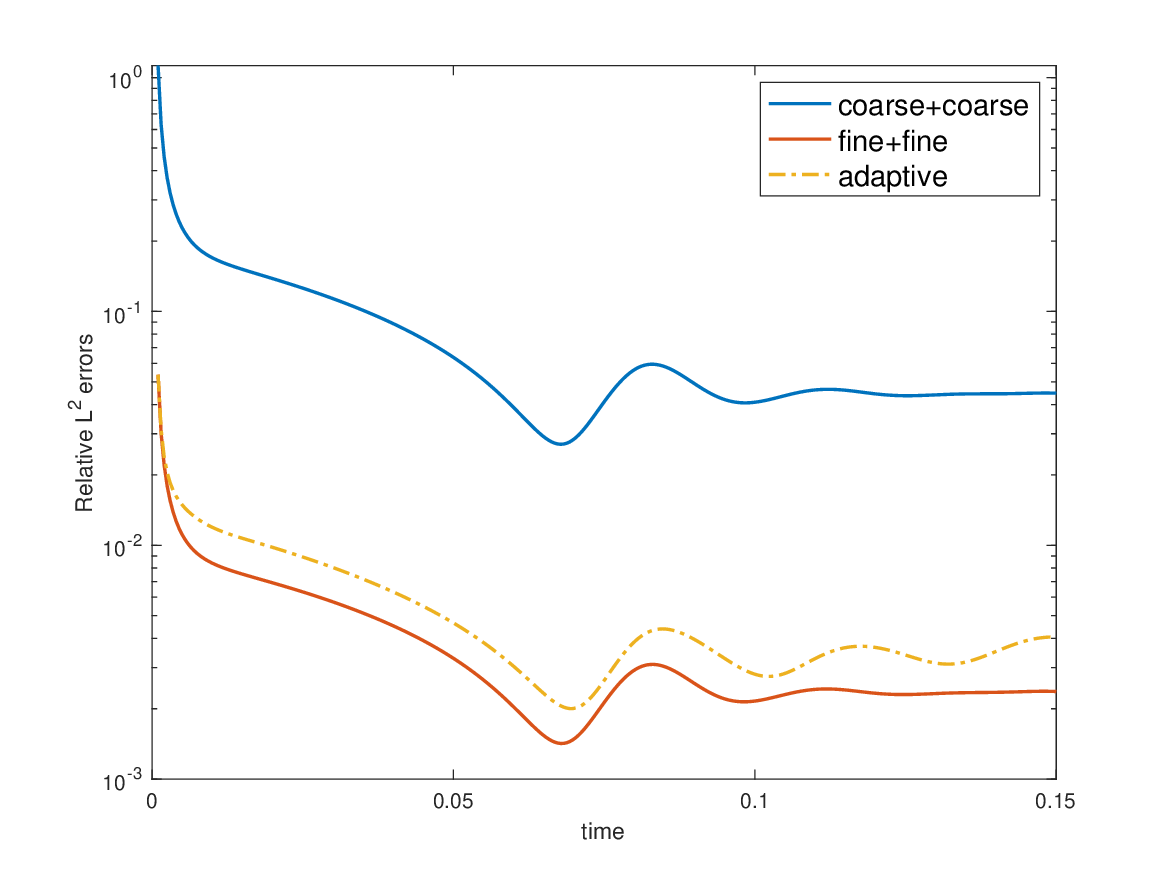}
\includegraphics[width=0.3\linewidth]{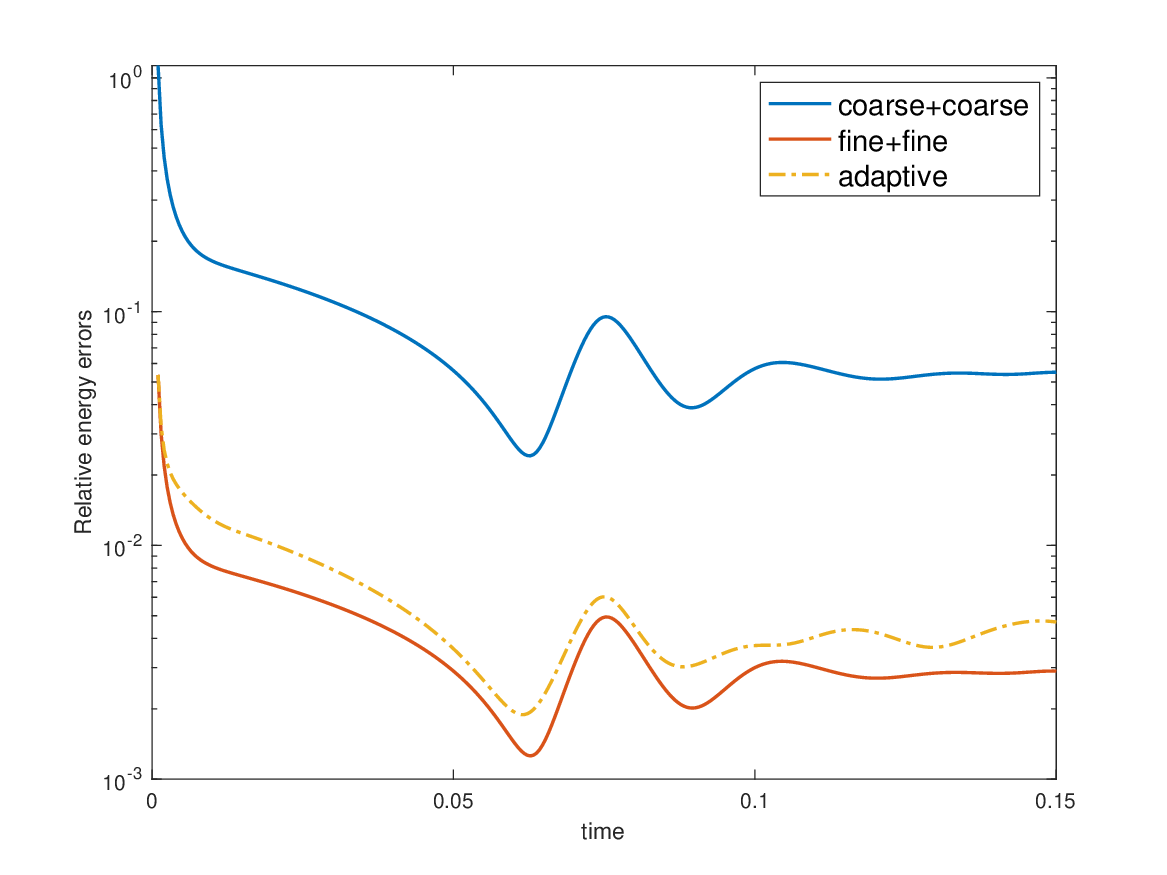}
\includegraphics[width=0.3\linewidth]{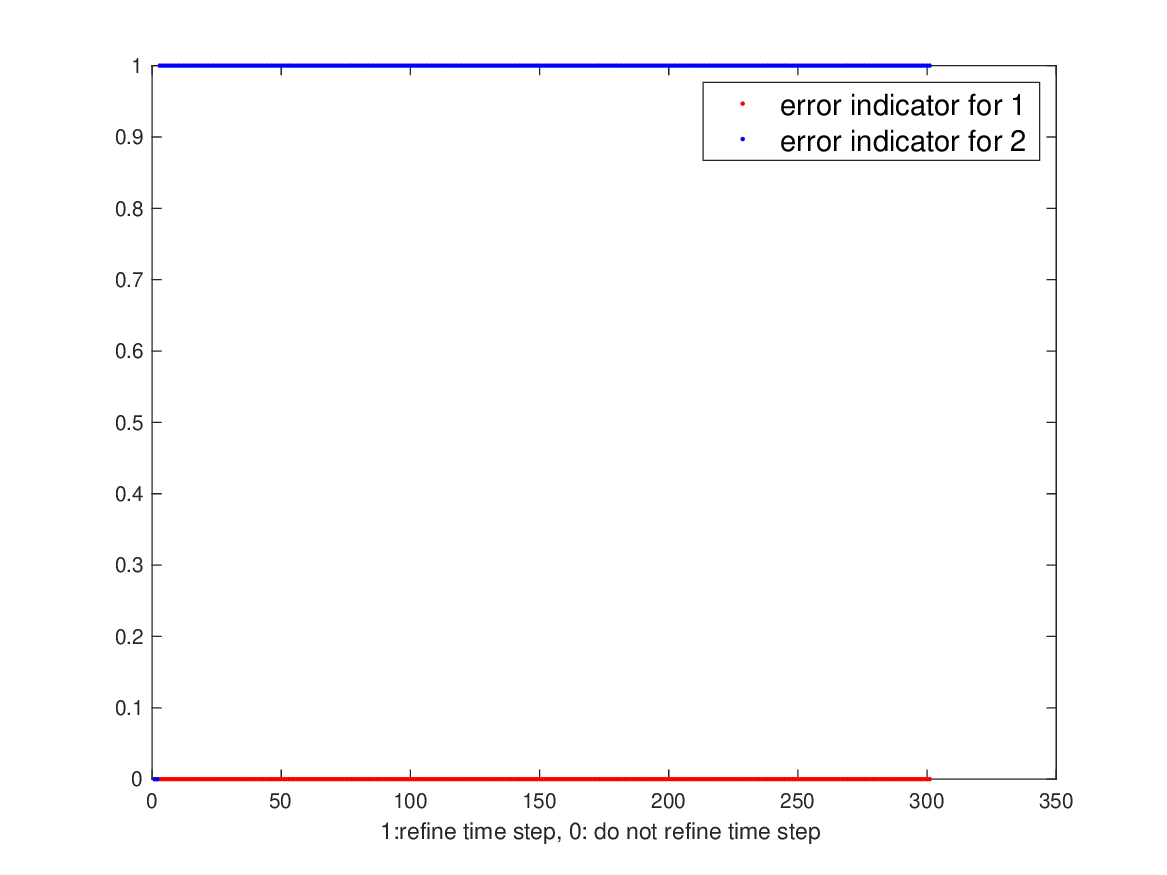}
\caption{Example 2. \textbf{Left:} relative $L_2$ error, \textbf{Middle:} relative energy error. \textbf{Right:} refinement history.}\label{ex2_errors}
\end{figure}

\subsection{Example 3: A three dimensional case}
\label{section: Numerical experiment 3}
In this example, we consider a three dimensional field, where the velocity field is a modified using the SEG Salt Model C3 wide azimuth classic dataset \cite{SaltModel}. The computational domain is adjusted to $\Omega = [0,1]^3$. We take $T=0.3$, $\Delta T=1\times10^{-4}$, $\Delta t=2\times10^{-5}$, $\delta t=1\times10^{-5}$, $h=0.02$, $H=0.1$. The degrees of freedom (DOF) for FEM is $132651$, and multiscale method is $8023$ where $\text{dim}(V_{aux,1}) = 23$ we take $l_i=8$ in $V_{aux,2}$. 
The source term is $f(t)f(x,y,z)$ where
\begin{align*}
    f(t)&=(1-2(\pi f_0 (t-t_0)^2)\exp(-(\pi f_0 (t-t_0))^2)\\
    f(x,y,z)&= \exp (-\frac{(x-x_0)^2+(y-y_0)^2+(z-z_0)^2}{2\sigma^2})
\end{align*}
where $(x_0,y_0,z_0)=(0.5,0.5,0.06), \sigma=0.03, f_0=5, t_0=0.22$.

We plot $\kappa$ and the solution at $T$ in Figure \ref{fig:ex3_kappa_f}. We remark that Dirichlet boundary condition is used in this example. Other practical cases including absorbing boundary conditions and PML can be considered in our future work.
\begin{figure}[H]
\centering
\includegraphics[width=0.45\linewidth]{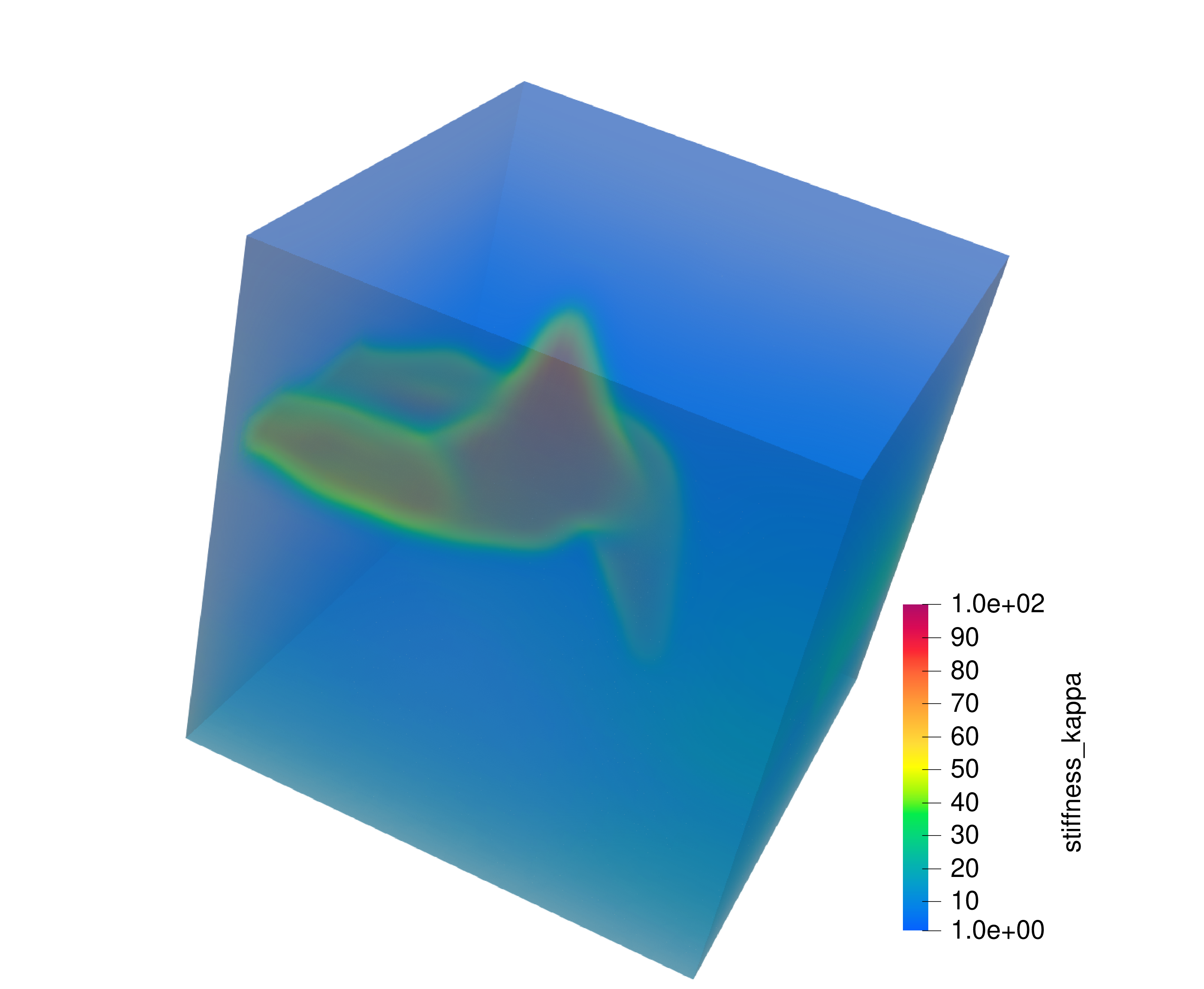}
\includegraphics[width=0.45\linewidth]{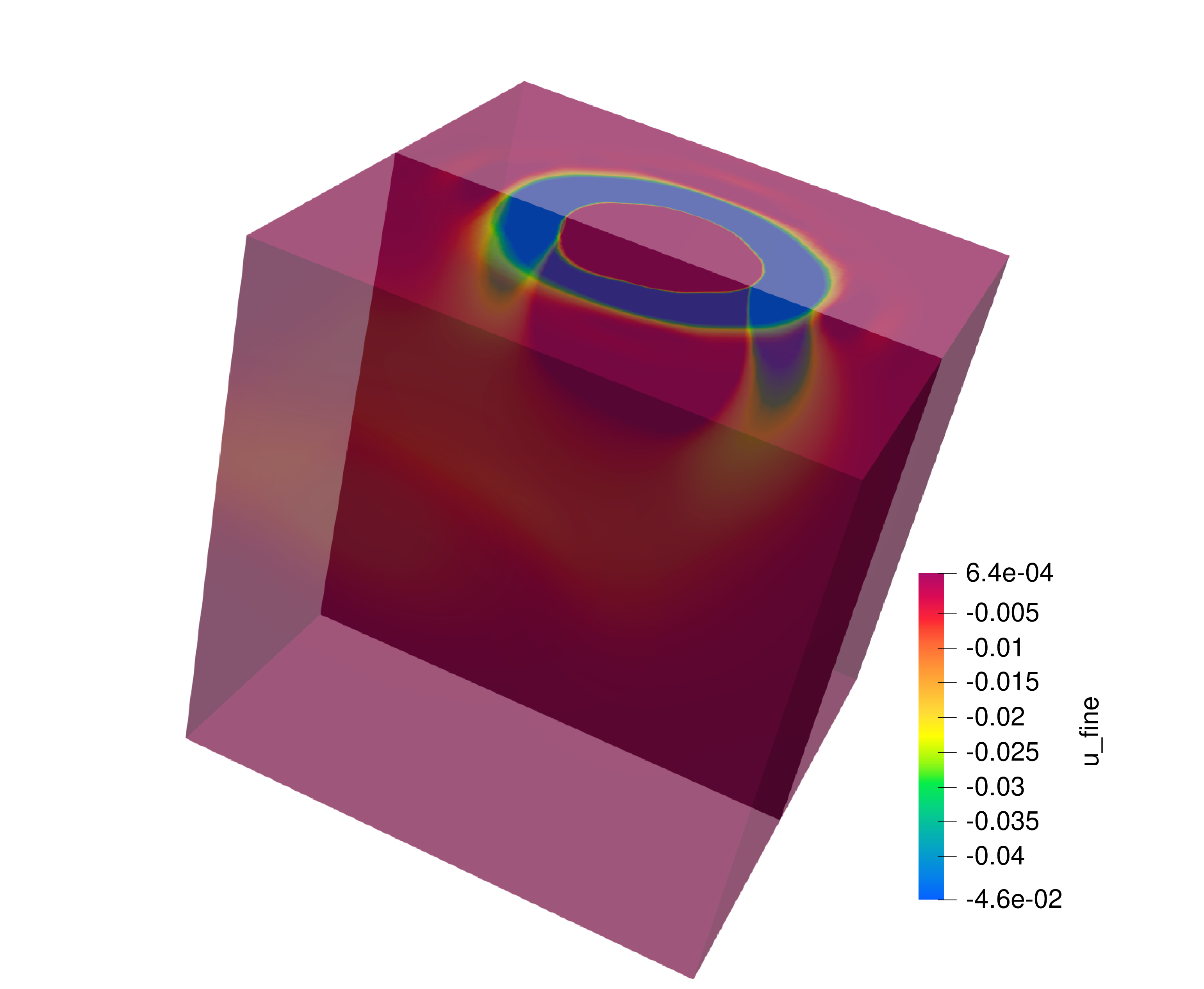}
\caption{Example 3. Left: $\kappa$. Right: the solution at $T$.}\label{fig:ex3_kappa_f}
\end{figure}

In the left of Figure \ref{fig:ex3_errors}, we show the error between the coarse solution and the FEM solution, when the mass lumping scheme is used or not used in the coarse scale system. It can be seen that the mass lumping error is relatively low, but the computational cost can be saved significantly in parctice.  We then compute the errors between the splitting scheme using uniform coarse time steps, using fine time steps and using multi-time stepping (in this example, we directly take `coarse-fine') with reference solutions, and the errors are presented in the middle and right of Figure \ref{fig:ex3_errors}. We see that the `coarse-fine' solution almost conincide with the `fine-fine' solution. For computational cost, we solve the linear systems in this example use PCG solvers. The wall time for fine scale finite element with implicit Euler scheme is around $4.5$ hours, and for CEM with implicit Euler with the same number of time steps is $2673.28$ seconds. However, our mass-lumped multi-rate partially explicit splitting scheme takes only $305.01$ seconds. This demonstrate that our proposed algorithm can achieve accuracy comparable to that of fine-fine solutions while reducing computational cost for the three dimensional case.

\begin{figure}[H]
\centering
\includegraphics[width=0.3\linewidth]{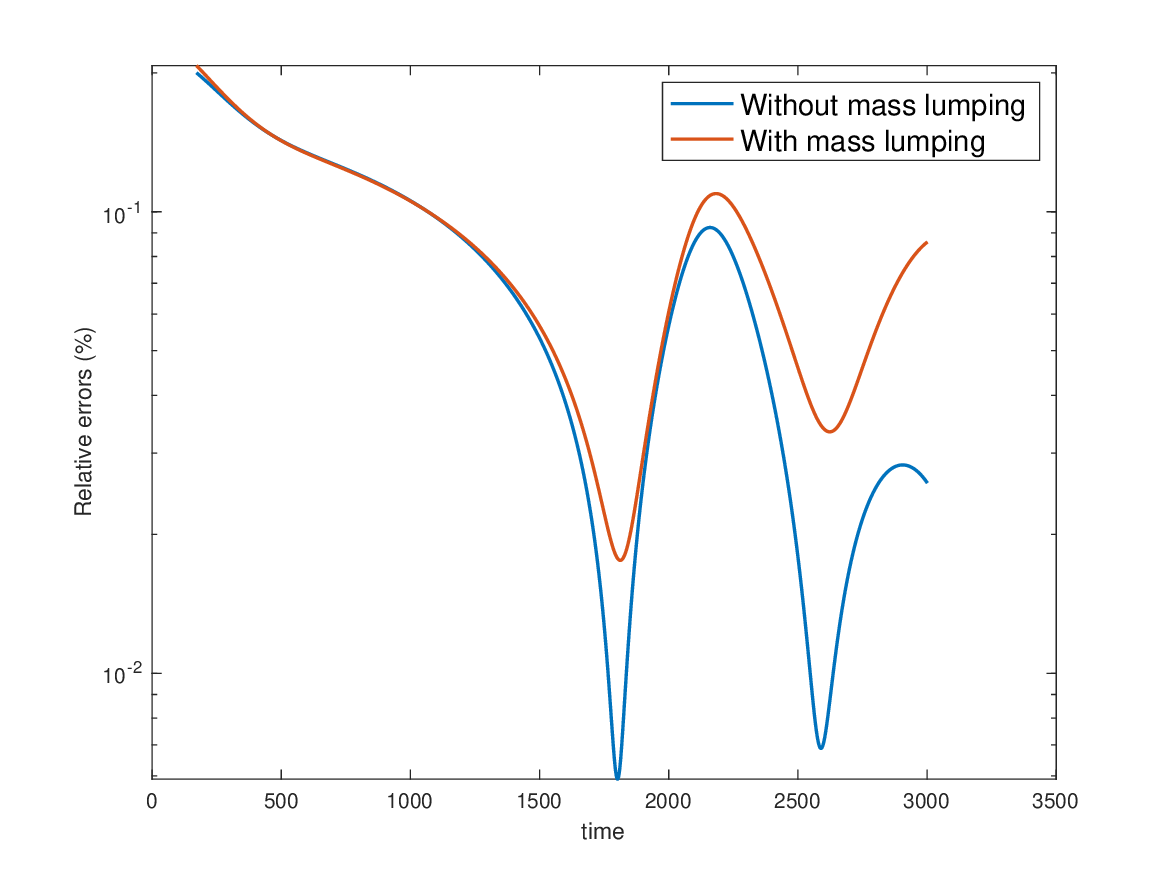}
\includegraphics[width=0.3\linewidth]{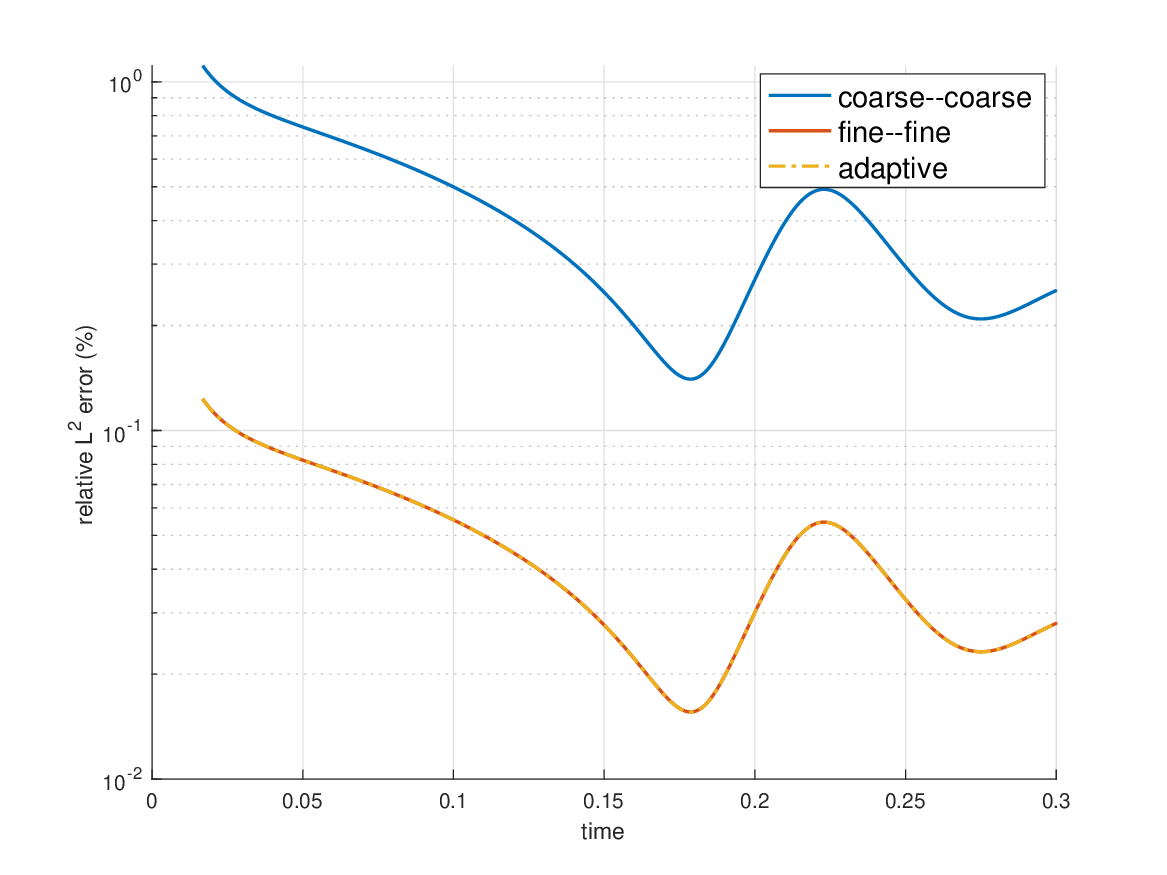}
\includegraphics[width=0.3\linewidth]{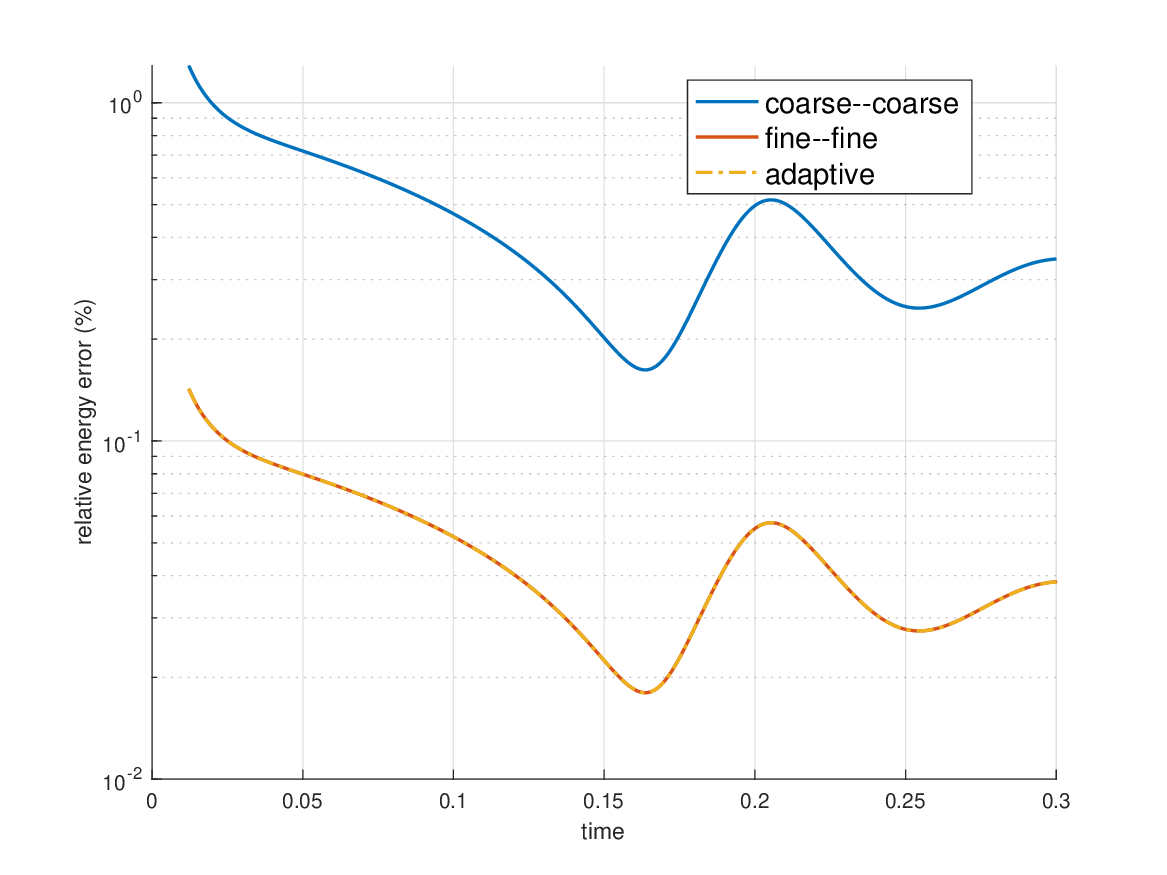}
\caption{Example 3. \textbf{Left:} mass lumping error, \textbf{Middle:} relative $L_2$ error. \textbf{Right:} relative energy error.}\label{fig:ex3_errors}
\end{figure}

\section{Concluding remarks}
\label{section: Concluding remarks}
In this work, we propose a multirate method for the partially explicit temporal splitting scheme to solve multiscale wave equations. First, the solution is decomposed into high contrast components and matrix components, which are respectively solved using implicit and explicit schemes. The resulting system is then decoupled through the mass lumping method. Next we introduce the multirate time splitting scheme, and analyze its stability results under different time step configurations. Finally we introduce a residual-based adaptive algorithm, derived the error indicators and provided an upper bound for the error. Numerical experiments demonstrate that the proposed algorithm adaptively refines the time step size based on the error indicators, thereby improving the approximation accuracy of the numerical solution while reducing computational cost.

\section*{Acknowledgments}
This work is supported by the Tianyuan Fund for Mathematics of the National Natural Science Foundation of China (Grant No. 12526212), the National Natural Science Foundation of China (NSFC Grant No. 12301559).

% \printbibliography
\section*{Appendix}\label{sec:appendix}
We give the proof of Theorem \ref{thm:adaptive} below.

\textbf{Proof:} 
% According to the definition of the residual, we have
% \begin{align*}
% \begin{aligned}
% &\int_{T_{n-1}}^{T_{n+1}}\left(R(U(t)),v\right)dt\\
% =&\int_{T_{n-1}}^{T_{n+1}}\biggl[(f(t),v)-(f^{n+1},v)-b(u''(t),v)-a(u(t),v)+\\
% &b(\frac{u_{H,1}^{n+1}-2u_{H,1}^{n}+u_{H,1}^{n-1}}{\Delta T^2},v_1)+b(\frac{u_{H,2}^{n+1}-2u_{H,2}^{n}+u_{H,2}^{n-1}}{\Delta T^2},v_2)\\
% &+a(u_{H,1}^{n}+u_{H,2}^{n},v_2)+\frac{1}{2}a(u_{H,1}^{n+1}+u_{H,1}^{n-1}+2u_{H,2}^{n},v_1)\biggr]dt.
% \end{aligned}
% \end{align*}
Since $b(u,v)=0$, $\forall v_1\in V_{H,1}$ and $v_2\in V_{H,2}$, then we get
\begin{align*}
b(u''(t),v)-b(\frac{u_{H,1}^{n+1}-2u_{H,1}^{n}+u_{H,1}^{n-1}}{\Delta T^2},v_1)-b(\frac{u_{H,2}^{n+1}-2u_{H,2}^{n}+u_{H,2}^{n-1}}{\Delta T^2},v_2)=0.
\end{align*}
The residual can then be written as follows
\begin{align*}
\int_{T_{n-1}}^{T_{n+1}}\left(R(U(t)),v\right)dt=\int_{T_{n-1}}^{T_{n+1}}\left((f(t),v)-(f^{n+1},v)+(R_1,v_1)+(R_2,v_2)\right)dt,
\end{align*}
where the term $\int_{T_{n-1}}^{T_{n+1}}(R_1,v_1)dt$
\begin{align*}
\begin{aligned}
=&\int_{T_{n-1}}^{T_{n+1}}\biggl[\frac{1}{2}a(u_{H,1}^{n+1}-2u_{H,1}^{n}+u_{H,1}^{n-1},v_1)-\frac{t-T_n}{2\Delta T}a(u_{H,1}^{n+1}-u_{H,1}^{n-1},v_1)\\
&-\frac{(t-T_n)^2}{2\Delta T^2}a(u_{H,1}^{n+1}-2u_{H,1}^{n}+u_{H,1}^{n-1},v_1)-\frac{t-T_n}{2\Delta T}a(u_{H,2}^{n+1}-u_{H,2}^{n-1},v_1) \\
&-\frac{(t-T_n)^2}{2\Delta T^2}a(u_{H,2}^{n+1}-2u_{H,2}^{n}+u_{H,2}^{n-1},v_1)\biggr]dt
\end{aligned}
\end{align*}
and the term $\int_{T_{n-1}}^{T_{n+1}}(R_2,v_2)dt$
\begin{align*}
\begin{aligned}
=&\int_{T_{n-1}}^{T_{n+1}}\biggl[-\frac{t-T_n}{2\Delta T}a(u_{H,1}^{n+1}-u_{H,1}^{n-1},v_1)-\frac{t-T_n}{2\Delta T}a(u_{H,2}^{n+1}-u_{H,2}^{n-1},v_1)\\
&-\frac{(t-T_n)^2}{2\Delta T^2}a(u_{H,1}^{n+1}-2u_{H,1}^{n}+u_{H,1}^{n-1},v_1)\\
&-\frac{(t-T_n)^2}{2\Delta T^2}a(u_{H,2}^{n+1}-2u_{H,2}^{n}+u_{H,2}^{n-1},v_1)\biggr]dt.
\end{aligned}
\end{align*}
With the help of the Cauchy-Schwarz inequality, the definition of two error indicators as in \eqref{G3} and by the definition of $X_n$ norm, we have
%we have the following estimates for \eqref{G7}
% \begin{align}\label{G9}
% \begin{aligned}
% \int_{T_{n-1}}^{T_{n+1}}(R_1,v_1)dt\leq &(\sqrt{\frac{\Delta T}{2}}+\sqrt{\frac{\Delta T}{6}}+\sqrt{\frac{\Delta T}{10}}) (\Vert u_{H,1}^{n+1}-u_{H,1}^{n}\Vert_a+\Vert u_{H,1}^{n}-u_{H,1}^{n-1}\Vert_a)\left(\int_{T_{n-1}}^{T_{n+1}}\Vert v_1\Vert_a^2dt\right)^{\frac{1}{2}}\\
% &+(\sqrt{\frac{\Delta T}{6}}+\sqrt{\frac{\Delta T}{10}}) (\Vert u_{H,2}^{n+1}-u_{H,2}^{n}\Vert_a+\Vert u_{H,2}^{n}-u_{H,2}^{n-1}\Vert_a)\left(\int_{T_{n-1}}^{T_{n+1}}\Vert v_1\Vert_a^2dt\right)^{\frac{1}{2}},
% \end{aligned}
% \end{align}
% and for \eqref{G8}
% \begin{align}\label{G10}
% \begin{aligned}
% \int_{T_{n-1}}^{T_{n+1}}(R_2,v_2)dt\leq&(\sqrt{\frac{\Delta T}{6}}+\sqrt{\frac{\Delta T}{10}}) (\Vert u_{H,1}^{n+1}-u_{H,1}^{n}\Vert_a+\Vert u_{H,1}^{n}-u_{H,1}^{n-1}\Vert_a\\
% &+\Vert u_{H,2}^{n+1}-u_{H,2}^{n}\Vert_a+\Vert u_{H,2}^{n}-u_{H,2}^{n-1}\Vert_a)\left(\int_{T_{n-1}}^{T_{n+1}}\Vert v_2\Vert_a^2dt\right)^{\frac{1}{2}}.
% \end{aligned}
% \end{align}

\begin{align*}
&\int_{T_{n-1}}^{T_{n+1}}(R_1,v_1)dt+\int_{T_{n-1}}^{T_{n+1}}(R_2,v_2)dt\\
\leq& C_1\eta_1^n\Vert v_1\Vert_{X_n}+C_2\eta_2^n\Vert v_2\Vert_{X_n}\leq C_3\left((\eta_1^n)^2+(\eta_2^n)^2\right)^{\frac{1}{2}}\left(\Vert v_1\Vert_{X_n}^2+\Vert v_2\Vert_{X_n}^2\right)^{\frac{1}{2}}.
\end{align*}
Note that
\begin{align*}
\Vert v_1+v_2\Vert_{X_n}^2\geq&\Vert v_1\Vert_{X_n}^2+\Vert v_2\Vert_{X_n}^2-2\gamma\Vert v_1\Vert_{X_n}\Vert v_2\Vert_{X_n}\\
\geq&\Vert v_1\Vert_{X_n}^2+\Vert v_2\Vert_{X_n}^2-\gamma(\Vert v_1\Vert_{X_n}^2+\Vert v_2\Vert_{X_n}^2)\\
=&(1-\gamma)(\Vert v_1\Vert_{X_n}^2+\Vert v_2\Vert_{X_n}^2),
\end{align*}
then we obtain
\begin{align*}
(\Vert v_1\Vert_{X_n}^2+\Vert v_2\Vert_{X_n}^2)^{\frac{1}{2}}\leq(1-\gamma)^{-\frac{1}{2}}\Vert v_1+v_2\Vert_{X_n}=(1-\gamma)^{-\frac{1}{2}}\Vert v\Vert_{X_n}.
\end{align*}
Therefore 
\begin{align*}
\int_{T_{n-1}}^{T_{n+1}}(R_1,v_1)dt+\int_{T_{n-1}}^{T_{n+1}}(R_2,v_2)dt\leq C_3(1-\gamma)^{-\frac{1}{2}}\left((\eta_1^n)^2+(\eta_2^n)^2\right)^{\frac{1}{2}}\Vert v\Vert_{X_n}.
\end{align*}
Finally we get
\begin{align*}
\mathcal{R}_v\leq&\int_{T_{n-1}}^{T_{n+1}}(f(t)-F(t),v)dt+C_3(1-\gamma)^{-\frac{1}{2}}\left((\eta_1^n)^2+(\eta_2^n)^2\right)^{\frac{1}{2}}\Vert v\Vert_{X_n}\\
% \leq&C_3(1-\gamma)^{-\frac{1}{2}}\left((\eta_1^n)^2+(\eta_2^n)^2\right)^{\frac{1}{2}}\Vert v\Vert_{X_n}+C_4\Vert f(t)-F(t)\Vert_{L^2((T_{n-1},T_{n+1}];L^2(\Omega))}\Vert v\Vert_{X_n}\\
\leq&C_3(1-\gamma)^{-\frac{1}{2}}\left((\eta_1^n)^2+(\eta_2^n)^2\right)^{\frac{1}{2}}+C_4\Vert f(t)-F(t)\Vert_{L^2((T_{n-1},T_{n+1}];L^2(\Omega))}.
\end{align*}
This completes the proof.

\bibliographystyle{plain}
\bibliography{lookup}
\end{document}